\documentclass[a4paper,12pt]{article}

\usepackage{amsmath}
\usepackage{ntheorem}
\usepackage{amssymb}
\usepackage{latexsym}
\usepackage[all]{xy}\UseComputerModernTips
\usepackage{graphicx}

\theoremstyle{plain}
\newtheorem{proposition}{Proposition}[section]
\newtheorem{theorem}[proposition]{Theorem}
\newtheorem{lemma}[proposition]{Lemma}
\newtheorem{corollary}[proposition]{Corollary}

\theoremstyle{plain}
\theorembodyfont{\normalfont\rm}
\newtheorem{definition}[proposition]{Definition}
\newtheorem{remark}[proposition]{Remark}

\theoremstyle{nonumberplain}
\theoremheaderfont{\normalfont\itshape}
\theoremseparator{.}
\theorembodyfont{\normalfont}
\newtheorem{proof}{Proof}

\newcommand{\qed}{\hfill $\Box$}

\makeatletter

\@addtoreset{equation}{section}
\makeatother

\newcommand{\ZZ}{{\mathbb Z}}
\newcommand{\RR}{{\mathbb R}}
\newcommand{\CC}{{\mathbb C}}

\newcommand{\LL}{{\mathbb L}}

\renewcommand{\d}{{\rm dim}}

\newcommand{\Vol}{{\rm Vol}}

\newcommand{\e}{\varepsilon}
\renewcommand{\SS}{{\mathcal S}}
\newcommand{\Spec}{{\rm Spec}}

\newcommand{\id}{{\rm id}}
\newcommand{\supp}{{\rm supp}}

\newcommand{\Int}{{\rm Int}}
\newcommand{\relint}{{\rm rel.int}}
\newcommand{\grad}{{\rm grad}}

\newcommand{\Db}{{\bf D}^{b}}
\newcommand{\Dbc}{{\bf D}_{c}^{b}}
\newcommand{\Kbc}{{\bf K}_{c}^{b}}

\newcommand{\F}{{\cal F}}
\newcommand{\G}{{\cal G}}

\renewcommand{\L}{{\cal L}}

\renewcommand{\1}{{\bf 1}}
\newcommand{\CF}{{\rm CF}}

\newcommand{\RG}{R\varGamma}

\newcommand{\tl}[1]{\widetilde{#1}}

\newcommand{\simto}{\overset{\sim}{\longrightarrow}}
\newcommand{\simot}{\overset{\sim}{\longleftarrow}}
\newcommand{\dsum}{\displaystyle \sum}

\renewcommand{\(}{\left(}
\renewcommand{\)}{\right)}

\newcommand{\longhookrightarrow}{\DOTSB\lhook\joinrel\longrightarrow}

\renewcommand{\labelenumi}{{\rm (\roman{enumi})}}

\title{Milnor fibers over singular toric varieties and nearby cycle sheaves}
\author{Yutaka \textsc{Matsui}\footnote{Department of Mathematics, Kinki University, 3-4-1, Kowakae, Higashi-Osaka, Osaka, 577-8502, Japan. E-mail: matsui@math.kindai.ac.jp} \and Kiyoshi \textsc{Takeuchi}\footnote{Institute of Mathematics, University  of Tsukuba, 1-1-1, Tennodai, Tsukuba, Ibaraki, 305-8571, Japan. E-mail: takemicro@nifty.com}}

\date{}

\begin{document}

\maketitle

\begin{abstract}
We propose a new sheaf-theoretical method for the calculation of the monodromy zeta functions of Milnor fibrations. As an application, classical formulas of Kushnirenko \cite{Kushnirenko} and Varchenko \cite{Varchenko} etc. concerning polynomials on $\CC^n$ will be generalized to polynomial functions on any toric variety.
\end{abstract}

\section{Introduction}\label{sec:1}

One of the most beautiful results in the theory of Milnor fibrations would be the formula for the (local) Milnor monodromy zeta functions obtained by Varchenko \cite{Varchenko} (see also \cite{Kushnirenko} and \cite{Oka-2} for the detail of this subject). In his formula, the Milnor monodromy zeta function $\zeta_f(t)\in \CC (t)^*$ at $0 \in \CC^n$ of a polynomial $f(x) \in \CC [x_1, x_2, \ldots , x_n]$ on $\CC^n$ such that $f(0)=0$ is expressed by the geometry of the Newton polygon of $f$ (for a similar and more precise result on Hodge structures, see also Tanabe \cite{Tanabe}). To prove it, he constructed a toric modification of $\CC^n$ on which the pull-back of $f$ defines a hypersurface with only normal crossing singularities. Since $\CC^n$ is a very special toric variety, it would be natural to generalize his formula to Milnor fibers over general singular toric varieties. In this paper, we realize this idea with the help of sheaf-theoretical methods, such as nearby cycle and constructible sheaves. In particular, in Theorem \ref{thm:3-4} we prove a formula for the monodromy zeta functions of Milnor fibers over general (not necessarily normal) toric varieties. Note that general theories of Milnor fibers over complete intersection varieties were developed by Looijenga \cite{Looijenga} and Oka \cite{Oka-2} etc. However toric varieties are not complete intersection nor of isolated singularities in general. Also for Milnor fibers over varieties of determinantal singularities, see Esterov \cite{Esterov}. 

In order to give the precise statement of our theorem, let $\SS$ be a finitely generated subsemigroup of the lattice $M \simeq \ZZ^n$ such that $0 \in \SS$. Denote by $K(\SS)$ the convex hull of $\SS$ in $M_{\RR} =\RR \otimes_{\ZZ}M$. For simplicity, assume that $K(\SS)$ is a strongly convex polyhedral cone in $M_{\RR}$ (for the general case, see Remark \ref{rem:3-7}) such that $\d K(\SS) =n$ and let $M(\SS)$ be the $\ZZ$-sublattice of rank $n$ in $M$ generated by $\SS$. Then $X(\SS) =\Spec(\CC[\SS])$ is a (not necessarily normal) toric variety of dimension $n$ (see \cite{Fulton}, \cite{G-K-Z} and \cite{Oda} etc. for the detail) on which the algebraic torus $T =\Spec(\CC[M(\SS)]) \simeq (\CC^*)^n$ acts. By our assumption, there exists a unique $T$-fixed point in $X(\SS)$, which we denote simply by $0$. Let $f \colon X(\SS) \longrightarrow \CC$ be a non-zero polynomial function on $X(\SS)$ (i.e. $f=\sum_{v \in \SS} a_v \cdot v$, $a_v \in \CC$) such that $f(0)=0$. Denote by $F_0$ the Milnor fiber of $f \colon X(\SS) \longrightarrow \CC$ at $0 \in X(\SS)$ (see for example \cite{Takeuchi} for a review on this subject). We define the monodromy zeta function $\zeta_{f,0}(t) \in \CC (t)^*$ of $f$ at $0 \in X(\SS)$ by 
\begin{equation}
\zeta_{f, 0}(t)=\prod_{j=0}^{\infty} \det(\id -t\Phi_{j,0})^{(-1)^j}, 
\end{equation}
where 
\begin{equation}
\Phi_{j,0} \colon H^j(F_0;\CC) \overset{\sim}{\longrightarrow} H^j(F_0;\CC) \qquad \ (j=0,1,\ldots)
\end{equation}
are the isomorphisms induced by the geometric monodromy automorphism $F_0 \simto F_0$. Then we can give a formula for the zeta function $\zeta_{f,0}(t)$ as follows. First, we define the Newton polygon $\Gamma_+(f) \subset K(\SS)$ of $f$ just as in the classical case of polynomials on $\CC^n$ (see Definition \ref{dfn:3-1}). For each face $\Delta \prec K(\SS)$ of the cone $K(\SS)$ such that $\Gamma_+(f) \cap \Delta \neq \emptyset$, let $\gamma_1^{\Delta}, \gamma_2^{\Delta},\ldots, \gamma_{n(\Delta)}^{\Delta}$ be the compact faces of $\Gamma_+(f) \cap \Delta$ such that $\d \gamma_i^{\Delta}=\d \Delta -1$. Let $\LL(\Delta)$ be the linear subspace of $M_{\RR}$ spanned by $\Delta$ and denote by $M(\SS \cap \Delta )$ the sublattice of $M(\SS)$ generated by $\SS \cap \Delta$. Then we can define the lattice distance $d_i^{\Delta}\in \ZZ_{>0}$ from $\gamma_i^{\Delta}$ to $0 \in \LL(\Delta)$ with respect to the lattice $M(\SS \cap \Delta ) \subset \LL(\Delta)$ (see Definition \ref{dfn:3-3}). Finally, let $\Vol_{\ZZ}(\gamma_i^{\Delta}) \in \ZZ$ be the normalized ($\d \Delta -1$)-dimensional volume of $\gamma_i^{\Delta}$ with respect to the lattice $M(\SS \cap \Delta) \cap \LL(\gamma_i^{\Delta})$.

\begin{theorem}\label{thm:1-1} 
Assume that $f$ is non-degenerate (in the sense of Definition \ref{dfn:3-2} below). Then the monodromy zeta function $\zeta_{f,0}(t)$ of $f$ at $0 \in X(\SS)$ is given by
\begin{equation}
\zeta_{f,0} (t) =\prod_{\Gamma_+(f) \cap \Delta \neq \emptyset} \zeta_{\Delta}(t),
\end{equation}
where for each face $\Delta \prec K(\SS)$ of $K(\SS)$ such that $\Gamma_+(f) \cap \Delta \neq \emptyset$ we set
\begin{equation}
\zeta_{\Delta}(t) = \prod_{i=1}^{n(\Delta)} \left(1-t^{d_i^{\Delta}}\right)^{(-1)^{\d \Delta -1}\Vol_{\ZZ}(\gamma_i^{\Delta})}.
\end{equation}
\end{theorem}

We will prove this theorem by decomposing the problem into those on the closures of $T$-orbits in $X(\SS)$ with the help of nearby cycle functors introduced by Deligne \cite{Deligne} (see also \cite[Chapter VIII]{K-S} etc.). Recall the following basic correspondence ($0 \leq k \leq n$):
\begin{equation}
\{ \text{$k$-dimensional faces in $K(\SS)$} \}\overset{\text{1:1}}{\longleftrightarrow} \{ \text{$k$-dimensional $T$-orbits in $X(\SS)$}\}.
\end{equation}
For a face $\Delta$ of $K(\SS)$, denote by $T_{\Delta}$ the corresponding $T$-orbit in $X(\SS)$. Then we obtain a decomposition $X(\SS)=\bigsqcup_{\Delta \prec K(\SS)} T_{\Delta}$ of $X(\SS)$ into $T$-orbits. To prove Theorem \ref{thm:1-1}, we first interpret the classical notions of Milnor fibers into the language of nearby cycle sheaves and reduce the problem to the computation of the monodromy zeta functions of the nearby cycle sheaves $\psi_f(\CC_{T_{\Delta}})$ of the constructible sheaves $\CC_{T_{\Delta}}$ on $X(\SS)$. Then by Proposition \ref{prp:2-9} we can study the monodromy zeta function of $\psi_f(\CC_{T_{\Delta}})$ on the closure $\overline{T_{\Delta}}$ of $T_{\Delta}$. This simple idea largely simplifies the classical arguments and allows us to avoid topological difficulties we usually encounter in treating Milnor fibers over singular varieties. Indeed even the original proof of Varchenko's theorem in \cite{Varchenko} would be also simplified by our idea of decomposing $\CC^n$ into smaller tori $(\CC^*)^d$. Moreover, by applying the same idea to complete intersection subvarieties $\{ f_1=f_2= \cdots =f_k=0\}$ in $X(\SS)$, in Theorem \ref{thm:3-12} we obtain also a generalization of the deeper results of Kirillov \cite{Kirillov} and Oka \cite{Oka-1}, \cite{Oka-2} to Milnor fibers over complete intersection subvarieties of singular toric varieties. In our Theorem \ref{thm:3-12}, even on the smooth toric variety $\CC^n$ we could remove some technical assumptions (see \cite[Chapter IV, \S 4, page 205]{Oka-2}) imposed by \cite{Oka-1} and \cite{Oka-2}. For example, in our Theorem \ref{thm:3-12} we do not assume any condition on the Newton polygons of polynomial functions $f_1, f_2, \ldots, f_k$ on $X(\SS)$. Note that Theorem \ref{thm:3-12} is also a natural generalization of the formula for the local multiplicities of toric varieties in \cite[Theorem 3.16]{G-K-Z}. The proof of Theorem \ref{thm:3-12} is very simple and follows also from the functorial property (Proposition \ref{prp:2-9}) of the nearby cycle functor. Note that on $\CC^n$ Gusev \cite{Gusev} obtained independently a similar result in a special case as a corollary of his main result. In Section \ref{sec:5}, we extend our results to the monodromy zeta functions of $T$-invariant constructible sheaves. 

Finally, let us mention that the methods we developed in this paper can be applied also to other related problems. For example, in \cite{M-T-new3} we used this idea to compute the monodromy zeta functions at infinity. In another paper \cite{M-T-new1}, some applications of our methods to $A$-discriminant varieties are also given. 


\bigskip
\noindent{\bf Acknowledgements:} After submitting this paper to a preprint server, we were informed by Professor Gusev that he obtained a similar result on $\CC^n$. We thank him cordially for showing us his very interesting paper \cite{Gusev}.

\section{Preliminary notions and results}\label{sec:2}

In this section, we introduce basic notions and results which will be used in this paper. In this paper, we essentially follow the terminology of \cite{Dimca}, \cite{H-T-T} and \cite{K-S}. For example, for a topological space $X$ we denote by $\Db(X)$ the derived category whose objects are bounded complexes of sheaves of $\CC_X$-modules on $X$.

\begin{definition}\label{dfn:2-1}
Let $X$ be an algebraic variety over $\CC$. Then
\begin{enumerate}
\item We say that a sheaf $\F$ on $X$ is constructible if there exists a stratification $X=\bigsqcup_{\alpha} X_{\alpha}$ of $X$ such that $\F|_{X_{\alpha}}$ is a locally constant sheaf of finite rank for any $\alpha$.
\item We say that an object $\F$ of $\Db(X)$ is constructible if the cohomology sheaf $H^j(\F)$ of $\F$ is constructible for any $j \in \ZZ$. We denote by $\Dbc(X)$ the full subcategory of $\Db(X)$ consisting of constructible objects $\F$.
\end{enumerate}
\end{definition}

Recall that for any morphism $f \colon X \longrightarrow Y$ of algebraic varieties over $\CC$ there exists a functor
\begin{equation}
Rf_* \colon \Db(X) \longrightarrow \Db(Y)
\end{equation}
of direct images. This functor preserves the constructibility and we obtain also a functor
\begin{equation}
Rf_* \colon \Dbc(X) \longrightarrow \Dbc(Y).
\end{equation}
For other basic operations $Rf_!$, $f^{-1}$, $f^!$ etc. in derived categories, see \cite{K-S} for the detail.

Next we introduce the notion of constructible functions and explain its relation with that of constructible sheaves.

\begin{definition}\label{dfn:2-2}
Let $X$ be an algebraic variety over $\CC$ and $G$ an abelian group. Then we say a $G$-valued function $\rho \colon X \longrightarrow G$ on $X$ is constructible if there exists a stratification $X=\bigsqcup_{\alpha} X_{\alpha}$ of $X$ such that $\rho|_{X_{\alpha}}$ is constant for any $\alpha$. We denote by $\CF_G(X)$ the abelian group of $G$-valued constructible functions on $X$.
\end{definition}

Let $\CC(t)^*=\CC(t) \setminus \{0\}$ be the multiplicative group of the function field $\CC(t)$ of the scheme $\CC$. In this paper, we consider $\CF_G(X)$ only for $G=\ZZ$ or $\CC(t)^*$. For a $G$-valued constructible function $\rho \colon X \longrightarrow G$, by taking a stratification $X=\bigsqcup_{\alpha}X_{\alpha}$ of $X$ such that $\rho|_{X_{\alpha}}$ is constant for any $\alpha$ as above, we set
\begin{equation}
\int_X \rho :=\dsum_{\alpha}\chi(X_{\alpha}) \cdot \rho(x_{\alpha}) \in G,
\end{equation}
where $x_{\alpha}$ is a reference point in $X_{\alpha}$. Then we can easily show that $\int_X\rho \in G$ does not depend on the choice of the stratification $X=\bigsqcup_{\alpha} X_{\alpha}$ of $X$. Hence we obtain a homomorphism
\begin{equation}
\int_X \colon \CF_G(X) \longrightarrow G
\end{equation}
of abelian groups. For $\rho \in \CF_G(X)$, we call $\int_X \rho \in G$ the topological (Euler) integral of $\rho$ over $X$. More generally, for any morphism $f \colon X \longrightarrow Y$ of algebraic varieties over $\CC$ and $\rho \in \CF_G(X)$, we define the push-forward $\int_f \rho \in \CF_G(Y)$ of $\rho$ by
\begin{equation}
\( \int_f \rho \) (y):=\int_{f^{-1}(y)} \rho
\end{equation}
for $y \in Y$. This defines a homomorphism
\begin{equation}
\int_f \colon \CF_G(X) \longrightarrow \CF_G(Y)
\end{equation}
of abelian groups. If $G=\ZZ$, these operations $\int_X$ and $\int_f$ correspond to the ones $\RG(X;\ \cdot \ )$ and $Rf_*$ respectively in the derived categories as follows. For an algebraic variety $X$ over $\CC$, consider a free abelian group
\begin{equation}
\ZZ(\Dbc(X)):=\left\{ \left. \dsum_{j \colon \text{finite}}a_j [\F_j] \ \right| \ a_j \in \ZZ, \ \F_j \in \Dbc(X)\right\}
\end{equation}
generated by the objects $\F_j \in \Dbc(X)$ in $\Dbc(X)$ and take its subgroup
\begin{eqnarray}
R&:=&\langle [\F_2]-[\F_1]-[\F_3] \ | \F_1 \longrightarrow \F_2 \longrightarrow \F_3 \overset{+1}{\longrightarrow} \ \text{ is a distinguished triangle} \rangle\nonumber\\
&\subset &\ZZ(\Dbc(X)).
\end{eqnarray}
We set $\Kbc(X):=\ZZ(\Dbc(X))/R$ and call it the Grothendieck group of $\Dbc(X)$. Then the following result is well-known (see for example \cite[Theorem 9.7.1]{K-S}).

\begin{theorem}\label{thm:2-3}
The homomorphism
\begin{equation}
\chi_X \colon \Kbc(X) \longrightarrow \CF_{\ZZ}(X)
\end{equation}
defined by taking the local Euler-Poincar{\'e} indices:
\begin{equation}
\chi_X([\F])(x):=\dsum_{j \in \ZZ} (-1)^j \d_{\CC}H^j(\F)_x \hspace{5mm}(x 
\in X)
\end{equation}
is an isomorphism.
\end{theorem}

For any morphism $f \colon X \longrightarrow Y$ of algebraic varieties over $\CC$, there exists also a commutative diagram
\begin{equation}
\xymatrix{
\Kbc(X) \ar[r]^{Rf_*} \ar[d]^{\wr}_{\chi_X}& \Kbc(Y) \ar[d]^{\wr}_{\chi_Y} \\
\CF_{\ZZ}(X) \ar[r]^{\int_f} & \CF_{\ZZ}(Y).}
\end{equation}
In particular, if $Y$ is the one-point variety $\{{\rm pt}\}$ ($\Kbc(Y) \simeq \CF_{\ZZ}(Y) \simeq \ZZ$), we obtain a commutative diagram
\begin{equation}
\xymatrix@R=2.5mm@C=30mm{
\Kbc(X) \ar[rd]^{\chi(\RG(X; \ \cdot \ ))} \ar[dd]_{\wr}^{\chi_X}& \\
 & \ZZ. \\
\CF_{\ZZ}(X) \ar[ru]^{\int_X}}
\end{equation}

Among various operations in derived categories, the following nearby cycle functors introduced by Deligne will be frequently used in this paper (see \cite[Section 4.2]{Dimca} for an excellent survey of this subject).

\begin{definition}\label{dfn:2-4}
Let $f \colon X \longrightarrow \CC$ be a non-constant regular function on an algebraic variety $X$ over $\CC$. Set $X_0:= \{x\in X\ |\ f(x)=0\} \subset X$ and let $i_X \colon X_0 \longhookrightarrow X$, $j_X \colon X \setminus X_0 \longhookrightarrow X$ be inclusions. Let $p \colon \tl{\CC^*} \longrightarrow \CC^*$ be the universal covering of $\CC^* =\CC \setminus \{0\}$ ($\tl{\CC^*} \simeq \CC$) and consider the Cartesian square
\begin{equation}\label{eq:2-13}
\xymatrix@R=2.5mm@C=2.5mm{
\tl{X \setminus X_0} \ar[rr] \ar[dd]^{p_X} & &\tl{\CC^*} \ar[dd]^p \\
 & \Box & \\
X \setminus X_0 \ar[rr]^f & & \CC^*.}
\end{equation}
Then for $\F \in \Db(X)$ we set
\begin{equation}
\psi_f(\F) := i_X^{-1}R(j_X \circ p_X)_*(j_X \circ p_X)^{-1}\F \in \Db(X_0)
\end{equation}
and call it the nearby cycle of $\F$. 
\end{definition}

Since the nearby cycle functor preserves the constructibility, in the above situation we obtain a functor
\begin{equation}
\psi_f \colon \Dbc(X) \longrightarrow \Dbc(X_0).
\end{equation}

As we see in the next proposition, the nearby cycle functor $\psi_f$ generalizes the classical notion of Milnor fibers. First, let us recall the definition of Milnor fibers and Milnor monodromies over singular varieties (see for example \cite{Takeuchi} for a review on this subject). Let $X$ be a subvariety of $\CC^m$ and $f \colon X \longrightarrow \CC$ a non-constant regular function on $X$. Namely we assume that there exists a polynomial function $\tl{f} \colon \CC^m \longrightarrow \CC$ on $\CC^m$ such that $\tl{f}|_X=f$. For simplicity, assume also that the origin $0 \in \CC^m$ is contained in $X_0=\{x \in X \ |\ f(x)=0\}$. Then the following lemma is well-known (see for example \cite[Definition 1.4]{Massey}).

\begin{lemma}\label{lem:2-5}
For sufficiently small $\e >0$, there exists $\eta_0 >0$ with $0<\eta_0 \ll \e$ such that for $0 < \forall \eta <\eta_0$ the restriction of $f$:
\begin{equation}
X \cap B(0;\e) \cap \tl{f}^{-1}(D_{\eta}^*) \longrightarrow D_{\eta}^*
\end{equation}
is a topological fiber bundle over the punctured disk $D_{\eta}^*:=\{ z \in \CC \ |\ 0<|z|<\eta\}$, where $B(0;\e)$ is the open ball in $\CC^m$ with radius $\e$ centered at the origin.
\end{lemma}

\begin{definition}\label{dfn:2-6}
A fiber of the above fibration is called the Milnor fiber of the function $f \colon X\longrightarrow \CC$ at $0 \in X$ and we denote it by $F_0$.
\end{definition}

For $x \in X_0$, denote by $F_x$ the Milnor fiber of $f\colon X \longrightarrow \CC$ at $x$.

\begin{proposition}{\rm \bf(\cite[Proposition 4.2.2]{Dimca})}\label{prp:2-7-2} For any $\F\in \Dbc(X)$, $x \in X_0$ and $j \in \ZZ$, there exists a natural isomorphism
\begin{equation}\label{eq:2-24}
H^j(F_x ;\F) \simeq H^j(\psi_f(\F))_x. 
\end{equation}
\end{proposition}

By this proposition, we can study the cohomology groups $H^j(F_x;\CC)$ of the Milnor fiber $F_x$ by using sheaf theory. Recall also that in the above situation, as in the same way as the case of polynomial functions over $\CC^n$ (see \cite{Milnor}), we can define the Milnor monodromy operators
\begin{equation}
\Phi_{j,x} \colon H^j(F_x;\CC) \overset{\sim}{\longrightarrow} H^j(F_x;\CC) \ (j=0,1,\ldots)
\end{equation}
and the zeta-function
\begin{equation}
\zeta_{f,x}(t):=\prod_{j=0}^{\infty} \det(\id -t\Phi_{j,x})^{(-1)^j}
\end{equation}
associated with it. Since the above product is in fact finite, $\zeta_{f,x}(t)$ is a rational function of $t$ and its degree in $t$ is the topological Euler characteristic $\chi(F_x)$ of the Milnor fiber $F_x$. This classical notion of Milnor monodromy zeta functions can be also generalized as follows.

\begin{definition}\label{dfn:2-8}
Let $f \colon X \longrightarrow \CC$ be a non-constant regular function on $X$ and $\F \in \Dbc(X)$. Set $X_0 :=\{x\in X\ |\ f(x)=0\}$. Then there exists a monodromy automorphism
\begin{equation}
\Phi(\F) \colon \psi_f(\F) \simto \psi_f(\F)
\end{equation}
of $\psi_f(\F)$ in $\Dbc(X_0)$ associated with a generator of the group ${\rm Deck}(\tl{\CC^*}, \CC^*)\simeq \ZZ$ of the deck transformations of $p \colon \tl{\CC^*} \longrightarrow \CC^*$ in the diagram \eqref{eq:2-13}. We define a $\CC(t)^*$-valued constructible function $\zeta_f(\F) \in \CF_{\CC(t)^*}(X_0)$ on $X_0$ by
\begin{equation}
\zeta_{f,x}(\F)(t):=\prod_{j \in \ZZ} \det\(\id -t\Phi(\F)_{j,x}\)^{(-1)^j}
\end{equation}
for $x \in X_0$, where $\Phi(\F)_{j,x} \colon (H^j(\psi_f(\F)))_x \simto (H^j(\psi_f(\F)))_x$ is the stalk at $x \in X_0$ of the sheaf homomorphism
\begin{equation}
\Phi(\F)_j \colon H^j(\psi_f(\F)) \simto H^j(\psi_f(\F))
\end{equation}
associated with $\Phi(\F)$.
\end{definition}

The following proposition will play a crucial role in the proof of Theorem \ref{thm:3-4} and \ref{thm:3-12}. For the proof, see for example, \cite[p.170-173]{Dimca} and \cite{Schurmann}.

\begin{proposition}\label{prp:2-9}
Let $\pi \colon Y \longrightarrow X$ be a proper morphism of algebraic varieties over $\CC$ and $f \colon X \longrightarrow \CC$ a non-constant regular function on $X$. Set $g:=f \circ \pi \colon Y \longrightarrow \CC$, $X_0:=\{x\in X\ |\ f(x)=0\}$ and $Y_0:=\{y\in Y\ |\ g(y)=0\}=\pi^{-1}(X_0)$. Then for any $\G\in \Dbc(Y)$ we have
\begin{equation}
\int_{\pi|_{Y_0}} \zeta_g(\G) =\zeta_f(R\pi_*\G)
\end{equation}
in $\CF_{\CC(t)^*}(X_0)$, where
\begin{equation}
\int_{\pi|_{Y_0}}\colon \CF_{\CC(t)^*}(Y_0) \longrightarrow \CF_{\CC(t)^*}(X_0)
\end{equation}
is the push-forward of $\CC(t)^*$-valued constructible functions by $\pi|_{Y_0} \colon Y_0 \longrightarrow X_0$.
\end{proposition}

Finally, we recall Bernstein-Khovanskii-Kushnirenko's theorem \cite{Khovanskii}. 

\begin{definition}
Let $g_1, g_2, \ldots , g_p$ be Laurent polynomials on $(\CC^*)^n$. Then we say that the subvariety $Z^*=\{ x\in (\CC^*)^n \ |\ g_1(x)=g_2(x)= \cdots =g_p(x)=0 \}$ of $(\CC^*)^n$ is non-degenerate complete intersection if the $p$-form $dg_1 \wedge dg_2 \wedge \cdots \wedge dg_p$ does not vanish on it.
\end{definition}

\begin{definition}
Let $g(x)=\sum_{v \in \ZZ^n} a_vx^v$ be a Laurent polynomial on $(\CC^*)^n$ ($a_v\in \CC$). We call the convex hull of $\supp(g):=\{v\in \ZZ^n \ |\ a_v\neq 0\} \subset \ZZ^n \subset \RR^n$ in $\RR^n$ the Newton polygon of $g$ and denote it by $NP(g)$.
\end{definition}

\begin{theorem}[\cite{Khovanskii}]\label{thm:2-14}
Let $g_1, g_2, \ldots , g_p$ be Laurent polynomials on $(\CC^*)^n$. Assume that the subvariety $Z^*=\{ x\in (\CC^*)^n \ |\ g_1(x)=g_2(x)= \cdots =g_p(x)=0 \}$ of $(\CC^*)^n$ is non-degenerate complete intersection. Set $\Delta_i:=NP(g_i)$ for $i=1,\ldots, p$. Then we have
\begin{equation}
\chi(Z^*)=(-1)^{n-p}\dsum_{\begin{subarray}{c} a_1,\ldots,a_p \geq 1\\ a_1+\cdots +a_p=n\end{subarray}}\Vol_{\ZZ}(\underbrace{\Delta_1,\ldots,\Delta_1}_{\text{$a_1$-times}},\ldots,\underbrace{\Delta_p,\ldots,\Delta_p}_{\text{$a_p$-times}}), 
\end{equation}
where $\Vol_{\ZZ}(\underbrace{\Delta_1,\ldots,\Delta_1}_{\text{$a_1$-times}},\ldots,\underbrace{\Delta_p,\ldots,\Delta_p}_{\text{$a_p$-times}})\in \ZZ$ is the normalized $n$-dimensional mixed volume of $\underbrace{\Delta_1,\ldots,\Delta_1}_{\text{$a_1$-times}},\ldots,\underbrace{\Delta_p,\ldots,\Delta_p}_{\text{$a_p$-times}}$ with respect to the lattice $\ZZ^n \subset \RR^n$.
\end{theorem}

\begin{remark}\label{rem:2-13}
Let $Q_1,Q_2,\ldots,Q_n$ be integral polytopes in $(\RR^n, \ZZ^n)$. Then their normalized $n$-dimensional mixed volume $\Vol_{\ZZ}(Q_1,Q_2,\ldots,Q_n) \in \ZZ$ is given by the formula 
\begin{eqnarray}
\lefteqn{n! \Vol_{\ZZ}(Q_1, Q_2, \ldots , Q_n)}\nonumber\\
&=&\Vol_{\ZZ}(Q_1+ Q_2+ \cdots + Q_n)-\sum_{i=1}^n \Vol_{\ZZ}(Q_1+ \cdots +Q_{i-1}+Q_{i+1} + \cdots + Q_n)\nonumber\\
& &+\sum_{1 \leq i<j \leq n} \Vol_{\ZZ}(Q_1+ \cdots +Q_{i-1}+Q_{i+1} + \cdots +Q_{j-1}+Q_{j+1} + \cdots + Q_n)\nonumber\\
& &+\cdots + (-1)^{n-1} \sum_{i=1}^n \Vol_{\ZZ}(Q_i),
\end{eqnarray}
where $\Vol_{\ZZ}(\ \cdot\ )\in \ZZ$ is the normalized $n$-dimensional volume.
\end{remark}

\section{Milnor fibers over singular toric varieties}\label{sec:3}

In this section, we give explicit formulas for the monodromy zeta functions of non-degenerate polynomials over possibly singular toric varieties. These formulas can be considered to be natural generalizations of the fundamental results obtained by Kushnirenko \cite{Kushnirenko}, Varchenko \cite{Varchenko}, Kirillov \cite{Kirillov} and Oka \cite{Oka-1}, \cite{Oka-2} etc.

Let $M \simeq \ZZ^n$ be a $\ZZ$-lattice of rank $n$ and set $M_{\RR}:=\RR \otimes_{\ZZ}M$. We take a finitely generated subsemigroup $\SS$ of $M$ such that $0 \in \SS$ and denote by $K(\SS)$ the convex hull of $\SS$ in $M_{\RR}$. For simplicity, assume that $K(\SS)$ is a strongly convex polyhedral cone in $M_{\RR}$ (for the general case, see Remark \ref{rem:3-7}) such that $\d K(\SS) =n$. Then the group algebra $\CC[\SS]$ is finitely generated over $\CC$ and $X(\SS):=\Spec(\CC[\SS])$ is a (not necessarily normal) toric variety of dimension $n$ (see \cite{Fulton}, \cite{G-K-Z} and \cite{Oda} etc. for the detail). Indeed, let $M(\SS)$ be the $\ZZ$-sublattice of rank $n$ in $M$ generated by $\SS$ and consider the algebraic torus $T:=\Spec(\CC[M(\SS)]) \simeq (\CC^*)^n$. Then the affine toric variety $X(\SS)$ admits a natural action of $T=\Spec(\CC[M(\SS)])$ and has a unique $0$-dimensional orbit. We denote this orbit point by $0$ and call it the $T$-fixed point of $X(\SS)$. Recall that a polynomial function $f \colon X(\SS) \longrightarrow \CC$ on $X(\SS)$ corresponds to an element $f=\sum_{v \in \SS} a_v \cdot v$ ($a_v \in \CC$) of $\CC[\SS]$.

\begin{definition}\label{dfn:3-1}
Let $f =\sum_{v \in \SS} a_v \cdot v$ ($a_v \in \CC$) be a polynomial function on $X(\SS)$.
\begin{enumerate}\renewcommand{\labelenumi}{{\rm (\roman{enumi})}}
\item We define the support $\supp(f)$ of $f$ by
\begin{equation}
\supp(f) :=\{ v \in \SS \ |\ a_v \neq 0\} \subset \SS .
\end{equation}
\item We define the Newton polygon $\Gamma_+(f)$ of $f$ to be the convex hull of $\bigcup_{v \in \supp(f)}(v+ K(\SS))$ in $K(\SS)$.
\end{enumerate}
\end{definition}

Now let us fix a function $f\in \CC[\SS]$ such that $0 \notin \supp(f)$ (i.e. $f \colon X(\SS) \longrightarrow \CC$ vanishes at the $T$-fixed point $0$) and consider its Milnor fiber $F_0$ at $0 \in X(\SS)$. Choose a $\ZZ$-basis of $M(\SS)$ and identify $M(\SS)$ with $\ZZ^n$. Then each element $v$ of $\SS \subset M(\SS)$ is identified with a $\ZZ$-vector $v=(v_1,\ldots,v_n)$ and to any $g=\sum_{v \in \SS}b_v \cdot v \in \CC[\SS]$ we can associate a Laurent polynomial $L(g)=\sum_{v \in \SS}b_v \cdot x^v$ on $T=(\CC^*)^n$. One can easily prove that the following definition does not depend on the choice of the $\ZZ$-basis of $M(\SS)$.

\begin{definition}\label{dfn:3-2}
We say that $f=\sum_{v \in \SS}a_v \cdot v \in \CC[\SS]$ is non-degenerate if for any compact face $\gamma$ of $\Gamma_+(f)$ the complex hypersurface
\begin{equation}
\{ x=(x_1,\ldots,x_n) \in (\CC^*)^n \ |\ L(f_{\gamma}) (x)=0\}
\end{equation}
in $(\CC^*)^n$ is smooth and reduced, where we set $f_{\gamma}:=\sum_{v \in \gamma\cap \SS}a_v \cdot v$.
\end{definition}

For each face $\Delta \prec K(\SS)$ of $K(\SS)$ such that $\Gamma_+(f) \cap \Delta \neq \emptyset$, let $\gamma_1^{\Delta}, \gamma_2^{\Delta},\ldots, \gamma_{n(\Delta)}^{\Delta}$ be the compact faces of $\Gamma_+(f) \cap \Delta$ such that $\d \gamma_i^{\Delta}=\d \Delta -1$. Let $\LL(\Delta)$ be the linear subspace of $M_{\RR}$ spanned by $\Delta$ and denote by $M(\SS \cap \Delta )$ the sublattice of $M(\SS)$ generated by $\SS \cap \Delta$. Note that the rank of $M(\SS \cap \Delta)$ is $\d \Delta$ and we have $M(\SS \cap \Delta)_{\RR}=\RR \otimes_{\ZZ} M(\SS \cap \Delta) \simeq \LL(\Delta)$. Then there exists a unique primitive vector $u_i^{\Delta}$ in the dual lattice $M(\SS \cap \Delta)^*$ of $M(\SS \cap \Delta)$ which takes its minimal in $\Gamma_+(f) \cap \Delta$ exactly on $\gamma_i^{\Delta} \subset \Gamma_+(f) \cap \Delta$. 

\begin{definition}\label{dfn:3-3}
We define the lattice distance $d_i^{\Delta}\in \ZZ_{>0}$ from $\gamma_i^{\Delta}$ to the origin $0 \in \LL(\Delta)$ to be the value of $u_i^{\Delta} $ on $\gamma_i^{\Delta}$.
\end{definition}

Then by using the normalized $(\d \Delta -1)$-dimensional volume $\Vol_{\ZZ}(\gamma_i^{\Delta}) \in \ZZ$ of $\gamma_i^{\Delta}$ with respect to the lattice $M(\SS \cap \Delta) \cap \LL(\gamma_i^{\Delta})$ we have the following result. 

\begin{theorem}\label{thm:3-4} Assume that $f=\sum_{v \in \SS}a_v \cdot v \in \CC[\SS]$ is non-degenerate. Then the monodromy zeta function $\zeta_{f,0}(t)$ of $f \colon X(\SS) \longrightarrow \CC$ at $0 \in X(\SS)$ is given by
\begin{equation}
\zeta_{f,0}(t) =\prod_{\Gamma_+(f) \cap \Delta \neq \emptyset} \zeta_{\Delta}(t), 
\end{equation}
where for each face $\Delta \prec K(\SS)$ of $K(\SS)$ such that $\Gamma_+(f) \cap \Delta \neq \emptyset$ we set 
\begin{equation}
\zeta_{\Delta}(t) = \prod_{i=1}^{n(\Delta)} \left(1-t^{d_i^{\Delta}}\right)^{(-1)^{\d \Delta -1}\Vol_{\ZZ}(\gamma_i^{\Delta})}. 
\end{equation}
\end{theorem}

We will prove this theorem as the special case of Theorem \ref{thm:3-12} (see Section \ref{sec:4}).

Let $\Gamma_i^{\Delta}$ be the convex hull of $\gamma_i^{\Delta} \sqcup \{0\}$ in $\LL(\Delta)$. Then the normalized $(\d \Delta)$-dimensional volume $\Vol_{\ZZ}(\Gamma_i^{\Delta})\in \ZZ$ of $\Gamma_i^{\Delta}$ with respect to the lattice $M(\SS \cap \Delta)$ is equal to $d_i^{\Delta} \cdot \Vol_{\ZZ}(\gamma_i^{\Delta})$ and we obtain the following result.

\begin{corollary}\label{cor:3-5} Assume that $f=\sum_{v \in \SS}a_v \cdot v \in \CC[\SS]$ is non-degenerate. Then the Euler characteristic of the Milnor fiber $F_0$ of $f \colon X(\SS) \longrightarrow \CC$ at $0 \in X(\SS)$ is given by
\begin{equation}
\chi (F_0)=\sum_{\Gamma_+(f) \cap \Delta \neq \emptyset} (-1)^{\d\Delta -1} \sum_{i=1}^{n(\Delta)} \Vol_{\ZZ}(\Gamma_i^{\Delta}).
\end{equation}
\end{corollary}

Now recall the following correspondence ($0 \leq k \leq n$):
\begin{equation}
\{ \text{$k$-dimensional faces in $K(\SS)$} \}\overset{\text{1:1}}{\longleftrightarrow} \{ \text{$k$-dimensional $T$-orbits in $X(\SS)$}\}.
\end{equation}
For a face $\Delta$ of $K(\SS)$, we denote by $T_{\Delta}$ the corresponding $T$-orbit in $X(\SS)$. Namely we set $T_{\Delta}:=\Spec(\CC[M(\SS\cap\Delta)])$. Then we obtain a decomposition $X(\SS)=\bigsqcup_{\Delta \prec K(\SS)} T_{\Delta}$ of $X(\SS)$. By the proof of Theorem \ref{thm:3-12} below, we obtain the following local version of Bernstein-Khovanskii-Kushnirenko's theorem which expresses the Euler characteristic $\chi (T_{\Delta} \cap F_0)$ of $T_{\Delta} \cap F_0$ in terms of the Newton polygon of $f$.

\begin{corollary}\label{cor:3-6} Assume that $f=\sum_{v \in \SS}a_v \cdot v \in \CC[\SS]$ is non-degenerate. Then we have
\begin{eqnarray}
\chi (T_{\Delta} \cap F_0)
&=& \chi(\RG(F_0;\CC_{T_{\Delta} \cap F_0}))\\
&=& \chi(\psi_f(\CC_{T_{\Delta}})_0)\\
&=& (-1)^{\d \Delta -1} \sum_{i=1}^{n(\Delta)}\Vol_{\ZZ}(\Gamma_i^{\Delta}).
\end{eqnarray}
\end{corollary}

\begin{proof}
In the proof of Theorem \ref{thm:3-12} below, we will prove that
\begin{equation}
\chi(\psi_f(\CC_{T_{\Delta}})_0)= (-1)^{\d \Delta -1} \sum_{i=1}^{n(\Delta)}\Vol_{\ZZ}(\Gamma_i^{\Delta}).
\end{equation}
Moreover by using Proposition \ref{prp:2-7-2} we see easily that $\chi(\RG(F_0;\CC_{T_{\Delta} \cap F_0}))$ is equal to $\chi(\psi_f(\CC_{T_{\Delta}})_0)$. Since $T_{\Delta}$ is a $T$-orbit, the decomposition $X(\SS)=\bigsqcup_{\Delta \prec K(\SS)} T_{\Delta}$ of $X(\SS)$ satisfies the Whitney regularity condition along $T_{\Delta}$. Then the decomposition $F_0 =\bigsqcup_{\Delta \prec K(\SS)} (T_{\Delta} \cap F_0) $ of $F_0 \subset X(\SS)$ is also a Whitney stratification ($f \colon X(\SS) \longrightarrow \CC$ has the isolated stratified critical value $0 \in \CC$ by \cite[Proposition 1.3]{Massey}). Finally, by applying \cite[Theorem 4.1.22]{Dimca} to the constructible sheaf $\CC_{T_{\Delta} \cap F_0}$ on the Whitney stratified analytic space $F_0 =\bigsqcup_{\Delta \prec K(\SS)} (T_{\Delta} \cap F_0)$ we obtain
\begin{equation}
\chi(\RG(F_0;\CC_{T_{\Delta} \cap F_0}))=\chi (T_{\Delta} \cap F_0).
\end{equation}
This completes the proof. 
\qed
\end{proof}

\begin{remark}\label{rem:3-7}
In Theorem \ref{thm:3-4}, we assumed that $K(\SS)$ is strongly convex and $0 \in X(\SS)=\Spec(\CC[\SS])$ is the $T$-fixed point. We can remove these assumptions as follows. Let $\SS$ be a finitely generated subsemigroup of $M=\ZZ^n$ such that $0 \in \SS$ and $\d K(\SS)=n$. We shall explain how to calculate the monodromy zeta functions $\zeta_{f,x}(t)$ of polynomials $f$ on $X(\SS)$ at general points $x \in X(\SS)$. For a point $x$ of $X(\SS)$, let $\Delta_0$ be the unique face of $K(\SS)$ such that $x \in T_{\Delta_0}=\Spec(\CC[M(\SS\cap\Delta_0)])$. Then
\begin{equation}
Y(\SS):=\Spec(\CC[ \SS + M(\SS \cap \Delta_0)])
\end{equation}
is an open subset of $X(\SS)$ containing $T_{\Delta_0}$ and we can regard $f$ as a function on $Y(\SS)$. Set $K^{\prime}:=K(\SS + M(\SS \cap \Delta_0))$. Then there exists a decomposition $Y(\SS)=\bigsqcup_{\Delta \prec K^{\prime}} T_{\Delta}$ of $Y(\SS)$ into $T$-orbits. Note that $T_{\Delta_0}$ is the smallest $T$-orbit in $Y(\SS)$. For $\Delta \prec K^{\prime}$, let $i_{\Delta}\colon \overline{T_{\Delta}}=\Spec(\CC[(\SS \cap \Delta)+M(\SS \cap \Delta_0)]) \longhookrightarrow Y(\SS)$ be the embedding. Then by $Y(\SS)=\bigsqcup_{\Delta \prec K^{\prime}} T_{\Delta}$ and Proposition \ref{prp:2-9} we have
\begin{equation}
\zeta_{f,x}(t)=\prod_{\Delta \prec K^{\prime}}\zeta_{f \circ i_{\Delta}, x}(\CC_{T_{\Delta}})(t)
\end{equation}
(see also \eqref{eq:4-1}-\eqref{eq:4-5}). Now let us set $\SS_{\Delta}:=(\SS \cap \Delta)+ (M(\SS \cap \Delta) \cap \LL (\Delta_0))$ and $\tl{Z} (\Delta):=\Spec (\CC[\SS_{\Delta}])$. Then there exists a decomposition $\tl{Z}(\Delta)=\bigsqcup_{\Delta_1 \prec \Delta } T^{\prime}_{\Delta_1}$ of $\tl{Z}(\Delta)$ into $T$-orbits and the natural morphism $\pi_{\Delta}\colon \tl{Z}(\Delta)\longrightarrow \overline{T_{\Delta}}$ induces a finite covering $T^{\prime}_{\Delta_1}\longrightarrow T_{\Delta_1}$ for any $\Delta_1 \prec \Delta$. Since $\pi_{\Delta}$ induces an isomorphism $T^{\prime}_{\Delta} \simto T_{\Delta}$, we have $R\pi_{\Delta *}(\CC_{T^{\prime}_{\Delta}}) \simeq \CC_{T_{\Delta}}$. Therefore by Proposition \ref{prp:2-9}, in order to calculate $\zeta_{f \circ i_{\Delta}, x}(\CC_{T_{\Delta}})(t)$ it suffices to calculate $\zeta_{f \circ i_{\Delta} \circ \pi_{\Delta}}(\CC_{T^{\prime}_{\Delta}})(t)$ at each point of the finite set $\pi_{\Delta}^{-1}(x)=\{ p_1,p_2, \ldots , p_k\} \subset \tl{Z}(\Delta)$. Let $M^{\prime} \simeq \ZZ^{n -\d \Delta_0}$ be a sublattice of $M$ such that $M^{\prime} \oplus (M \cap \LL(\Delta_0))=M$ and set $\SS^{\prime}_{\Delta}=M^{\prime} \cap \SS_{\Delta}$. Then we have
\begin{equation}
\SS_{\Delta}= \SS^{\prime}_{\Delta}\oplus (M(\SS \cap \Delta) \cap \LL(\Delta_0))
\end{equation}
and $K(\SS^{\prime}_{\Delta})\subset M^{\prime}_{\RR}$ is a strongly convex cone. Hence we have
\begin{equation}
\tl{Z}(\Delta)\simeq\Spec (\CC [\SS^{\prime}_{\Delta}]) \times (\CC^*)^{\d \Delta_0}\supset \{ 0\} \times (\CC^*)^{\d \Delta_0}\supset \pi_{\Delta}^{-1}(x)
\end{equation}
and $\zeta_{f \circ i_{\Delta} \circ \pi_{\Delta}}(\CC_{T^{\prime}_{\Delta}})(t)$ can be calculated at each point of $\pi_{\Delta}^{-1}(x)=\{ p_1,p_2, \ldots , p_k\}$ by Corollary \ref{cor:3-6}. Indeed, we first multiply a monomial in $\CC[M(\SS \cap \Delta) \cap \LL(\Delta_0)] \subset \CC [\SS_{\Delta}]$ to $f \circ i_{\Delta} \circ \pi_{\Delta} \in \CC [\SS_{\Delta}]$ and extend it to a function on $\Spec (\CC [\SS^{\prime}_{\Delta}]) \times \CC^{\d \Delta_0}$. Then by a suitable translation we may assume that $p_i \in \pi_{\Delta}^{-1}(x)$ is the unique $T$-fixed point of the product toric variety $\Spec (\CC [\SS^{\prime}_{\Delta}]) \times \CC^{\d \Delta_0}$ and Corollary \ref{cor:3-6} can be applied. 
\end{remark}

In the rest of this section, we extend our results to non-degenerate complete intersection subvarieties in the affine toric variety $X(\SS)$. Let $f_1,f_2,\ldots, f_k \in \CC[\SS]$ ($1 \leq k \leq n=\d X(\SS)$) and consider the following subvarieties of $X(\SS)$:
\begin{equation}
V:=\{f_1=\cdots =f_{k-1}=f_k=0\} \subset W:=\{f_1=\cdots =f_{k-1}=0\}.
\end{equation}
Assume that $0 \in V$. Our objective here is to study the Milnor fiber $G_0$ of $g:=f_k|_W \colon W \longrightarrow \CC$ at $0 \in V=g^{-1}(0) \subset W$ and its monodromy zeta function $\zeta_{g,0}(t)$. We call $\zeta_{g,0}(t)$ the $k$-th principal monodromy zeta function of $V=\{ f_1=\cdots =f_k=0\}$. For each face $\Delta \prec K(\SS)$ of $K(\SS)$ such that $\Gamma_+(f_k) \cap \Delta \neq \emptyset$, we set
\begin{equation}
I(\Delta):=\{ j =1,2,\ldots, k-1 \ | \ \Gamma_+(f_j) \cap \Delta \neq \emptyset \} \subset \{ 1,2,\ldots, k-1\}
\end{equation}
and $m(\Delta):=\sharp I(\Delta)+1$. Let $\LL(\Delta)$, $M(\SS \cap \Delta)$, $M(\SS \cap \Delta)^*$ be as before and $\LL(\Delta)^*$ the dual vector space of $\LL(\Delta)$. Then $M(\SS \cap \Delta)^*$ is naturally identified with a subset of $\LL(\Delta)^*$ and the polar cone
\begin{equation}
\Delta^{\vee}=\{ u \in \LL(\Delta)^* \ | \ \text{$\langle u, v\rangle \geq 0$ for any $v \in \Delta$}\}
\end{equation}
of $\Delta$ in $\LL(\Delta)^*$ is a rational polyhedral convex cone with respect to the lattice $M(\SS \cap \Delta)^*$ in $\LL(\Delta)^*$.

\begin{definition}\label{dfn:3-7}
\begin{enumerate}
\item For a function $f=\sum_{v \in \Gamma_+(f)} a_v \cdot v \in \CC[\SS]$ on $X(\SS)$ and $u \in \Delta^{\vee}$, we set $f|_{\Delta}:= \sum_{v \in \Gamma_+(f) \cap \Delta}a_v \cdot v \in \CC[\SS \cap \Delta]$ and
\begin{equation}\label{eq:3-12}
\Gamma(f|_{\Delta};u):=\left\{ v\in \Gamma_+(f) \cap \Delta \ \left|\ \langle u, v \rangle =\min_{w \in \Gamma_+(f) \cap \Delta} \langle u,w \rangle \right.\right\}.
\end{equation}
We call $\Gamma(f|_{\Delta};u)$ the supporting face of $u$ in $\Gamma_+(f) \cap \Delta$.
\item For $j \in I(\Delta) \sqcup \{ k\}$ and $u \in \Delta^{\vee}$, we define the $u$-part $f_j^u \in \CC[\SS \cap \Delta]$ of $f_j$ by
\begin{equation}
f_j^u:=\sum_{v \in \Gamma(f_j|_{\Delta};u)} a_v \cdot v \in \CC[\SS \cap \Delta ],
\end{equation}
where $f_j=\sum_{v \in \Gamma_+(f_j)} a_v \cdot v $ in $\CC[\SS]$.
\end{enumerate}
\end{definition}

By taking a $\ZZ$-basis of $M(\SS)$ and identifying the $u$-parts $f_j^u$ with Laurent polynomials $L(f_j^u)$ on $T=(\CC^*)^n$ as before, we have the following definition which does not depend on the choice of the $\ZZ$-basis of $M(\SS)$.

\begin{definition}\label{dfn:3-9}
We say that $(f_1,\ldots, f_k)$ is non-degenerate if for any face $\Delta \prec K(\SS)$ such that $\Gamma_+(f_k) \cap \Delta \neq \emptyset$ (including the case where $\Delta =K(\SS)$) and any $u \in \Int(\Delta^{\vee}) \cap M(\SS \cap \Delta)^*$ the following two subvarieties of $(\CC^*)^n$ are non-degenerate complete intersections.
\begin{gather}
\{ x\in (\CC^*)^n \ |\ \text{$L(f_j^u)(x)=0$ for any $j \in I(\Delta)$}\},\\
\{ x\in (\CC^*)^n \ |\ \text{$L(f_j^u)(x)=0$ for any $j \in I(\Delta) \sqcup \{ k\}$}\}.
\end{gather}
\end{definition}

\begin{remark}\label{rem:3-10}
The above definition is slightly different from the one in \cite{Oka-2} etc., since our result (Theorem \ref{thm:3-12} below) generalizes the ones in \cite{Kirillov}, \cite{Oka-1} and \cite{Oka-2}.
\end{remark}

For each face $\Delta \prec K(\SS)$ of $K(\SS)$ such that $\Gamma_+(f_k) \cap \Delta \neq \emptyset$, let us set
\begin{equation}
f_{\Delta}:=\left(\prod_{j \in I(\Delta)}f_j\right) \cdot f_k \in \CC[\SS]
\end{equation}
and consider its Newton polygon $\Gamma_+(f_{\Delta}) \subset K(\SS)$. Let $\gamma_1^{\Delta}, \gamma_2^{\Delta}, \ldots, \gamma_{n(\Delta)}^{\Delta}$ be the compact faces of $\Gamma_+(f_{\Delta}) \cap \Delta$ ($\neq \emptyset$) such that $\d \gamma_i^{\Delta}=\d \Delta -1$. Then for each $1 \leq i \leq n(\Delta)$ there exists a unique primitive vector $u_i^{\Delta} \in \Int(\Delta^{\vee}) \cap M(\SS \cap \Delta)^*$ which takes its minimal in $\Gamma_+(f_{\Delta}) \cap \Delta$ exactly on $\gamma_i^{\Delta}$. For $j \in I(\Delta) \sqcup \{k\}$, set
\begin{equation}
\gamma(f_j)_i^{\Delta}:=\Gamma(f_j|_{\Delta}; u_i^{\Delta})
\end{equation}
and 
\begin{equation}
d_i^{\Delta} := \min_{w \in \Gamma_+(f_k) \cap \Delta} \langle u_i^{\Delta},w \rangle .
\end{equation}
Note that we have
\begin{equation}
\gamma_i^{\Delta}=\sum_{j \in I(\Delta)\sqcup \{ k\}} \gamma(f_j)_i^{\Delta}
\end{equation}
for any face $\Delta \prec K(\SS)$ such that $\Gamma_+(f_k) \cap \Delta \neq \emptyset$ and $1 \leq i \leq n(\Delta)$. For each face $\Delta \prec K(\SS)$ such that $\Gamma_+(f_k) \cap \Delta \neq \emptyset$, $\d \Delta \geq m(\Delta)$ and $1 \leq i \leq n(\Delta)$, we set
\begin{equation}
K_i^{\Delta}:=\hspace*{-10mm}\sum_{\begin{subarray}{c}\alpha_1+\cdots +\alpha_{m(\Delta)}=\d \Delta-1 \\ \text{$\alpha_q \geq 1$ for $q \leq m(\Delta)-1$}, \ \alpha_{m(\Delta)} \geq 0 \end{subarray}} \hspace*{-10mm}\Vol_{\ZZ}(\underbrace{\gamma(f_{j_1})_i^{\Delta},\ldots, \gamma(f_{j_1})_i^{\Delta}}_{\text{$\alpha_1$ times}}, \ldots, \underbrace{\gamma(f_{j_{m(\Delta)}} )_i^{\Delta}, \ldots, \gamma(f_{j_{m(\Delta)}} )_i^{\Delta}}_{\text{$\alpha_{m(\Delta)} $ times}}).
\end{equation}
Here we set $I(\Delta) \sqcup \{k\} =\{j_1, j_2, \ldots , k=j_{m(\Delta)} \}$ and
\begin{equation}
\Vol_{\ZZ}(\underbrace{\gamma(f_{j_1})_i^{\Delta},\ldots, \gamma(f_{j_1})_i^{\Delta}}_{\text{$\alpha_1$ times}}, \ldots, \underbrace{\gamma(f_{j_{m(\Delta)}} )_i^{\Delta}, \ldots, \gamma(f_{j_{m(\Delta)}} )_i^{\Delta}}_{\text{$\alpha_{m(\Delta)} $ times}})
\end{equation}
is the normalized $(\d \Delta -1)$-dimensional mixed volume of 
\begin{equation}
\underbrace{\gamma(f_{j_1})_i^{\Delta},\ldots, \gamma(f_{j_1})_i^{\Delta}}_{\text{$\alpha_1$ times}}, \ldots, \underbrace{\gamma(f_{j_{m(\Delta)}})_i^{\Delta}, \ldots, \gamma(f_{j_{m(\Delta)}})_i^{\Delta}}_{\text{$\alpha_{m(\Delta)}$ times}}
\end{equation}
(see Remark \ref{rem:2-13}) with respect to the lattice $M(\SS \cap \Delta)\cap \LL(\gamma_i^{\Delta})$.

\begin{remark}\label{rem:3-11}
If $\d \Delta -1=0$, we set
\begin{equation}
K_i^{\Delta}=\Vol_{\ZZ}(\underbrace{\gamma(f_k)_i^{\Delta}, \ldots, \gamma(f_k)_i^{\Delta}}_{\text{$0$ times}}):=1
\end{equation}
(in this case $\gamma(f_k)_i^{\Delta}$ is a point).
\end{remark}

\begin{theorem}\label{thm:3-12} Assume that $(f_1,\ldots, f_k)$ is non-degenerate. Then the $k$-th principal monodromy zeta function $\zeta_{g,0}(t)$ ($g=f_k|_W \colon W \longrightarrow \CC$) is given by
\begin{equation}
\zeta_{g,0}(t)= \prod_{\begin{subarray}{c} \Gamma_+(f_k) \cap \Delta \neq \emptyset \\ \d \Delta \geq m(\Delta) \end{subarray}} \zeta_{g,\Delta}(t), 
\end{equation}
where for each face $\Delta \prec K(\SS)$ of $K(\SS)$ such that $\Gamma_+(f_k) \cap \Delta \neq \emptyset$ and $\d \Delta \geq m(\Delta)$ we set 
\begin{equation}
\zeta_{g,\Delta}(t) = \prod_{i=1}^{n(\Delta)} \left(1-t^{d_i^{\Delta}} \right)^{(-1)^{\d \Delta -m(\Delta)}K_i^{\Delta}}.
\end{equation}
In particular, the Euler characteristic of the Milnor fiber $G_0$ of $g=f_k|_W \colon W \longrightarrow \CC$ at $0 \in V=g^{-1}(0)$ is given by
\begin{equation}\label{eq:3-28}
\chi(G_0)=\sum_{\begin{subarray}{c}
\Gamma_+(f_k) \cap \Delta \neq \emptyset \\ \d \Delta \geq m(\Delta) \end{subarray}} (-1)^{\d \Delta -m(\Delta)} \sum_{i=1}^{n(\Delta)} d_i^{\Delta}\cdot K_i^{\Delta}.
\end{equation}
\end{theorem}

\section{Proof of Theorem \ref{thm:3-12}}\label{sec:4}

Now let us prove Theorem \ref{thm:3-12}. Theorem \ref{thm:3-4} will be proved as a special case of Theorem \ref{thm:3-12}. Our proof is similar to the one in \cite{Varchenko}.

First let $i_{\Delta} \colon \overline{T_{\Delta}} \longhookrightarrow X(\SS)$ be the closed embedding. Since nearby cycle functors take distinguished triangles to distinguished triangles, by $X(\SS)=\bigsqcup_{\Delta \prec K(\SS)} T_{\Delta}$ and $W=\bigsqcup_{\Delta \prec K(\SS)}(T_{\Delta} \cap W)$ we obtain
\begin{eqnarray}
\label{eq:4-1}\zeta_{g, 0}(t)
&=&\zeta_{f_k,0}\(\CC_W\)(t)\\
\label{eq:4-2}&=&\zeta_{f_k,0}\(\CC_{W \setminus\{0\}}\)(t)\\
\label{eq:4-3}&=&\prod_{\{0\} \precneqq \Delta \prec K(\SS)} \zeta_{f_k,0}\(\CC_{T_{\Delta} \cap W}\)(t)\\
\label{eq:4-4}&=&\prod_{\{0\} \precneqq \Delta \prec K(\SS)} \zeta_{f_k,0}\(i_{\Delta *}\CC_{T_{\Delta} \cap W}\)(t)\\
\label{eq:4-5}&=&\prod_{\{0\}\precneqq \Delta \prec K(\SS)} \zeta_{f_k\circ i_{\Delta},0}\(\CC_{T_{\Delta} \cap W}\)(t). 
\end{eqnarray}
Here we used Proposition \ref{prp:2-9} to prove the first and last equalities. We set
\begin{equation}
\zeta_{g,\Delta}(t):=\zeta_{f_k\circ i_{\Delta},0}\(\CC_{T_{\Delta} \cap W}\)(t) \in \CC(t)^*.
\end{equation}
Since the condition $\Gamma_+(f_k) \cap \Delta =\emptyset$ is equivalent to the one $f_k \circ i_{\Delta}\equiv 0$, for a face $\Delta$ of $K(\SS)$ such that $\Gamma_+(f_k) \cap \Delta =\emptyset$ the nearby cycle $\psi_{f_k \circ i_{\Delta}}\(\CC_{T_{\Delta} \cap W}\)$ vanishes and hence $\zeta_{g,\Delta}(t) \equiv 1$. Therefore, in order to calculate the monodromy zeta function $\zeta_{g,0}(t)$ of $g=f_k|_W \colon W \longrightarrow \CC$ at $0 \in g^{-1}(0)$, it suffices to calculate $\zeta_{g,\Delta}(t)$ only for faces $\Delta$ of $K(\SS)$ such that $\Delta \neq \{0\}$ and $\Gamma_+(f_k) \cap \Delta \neq \emptyset$.

Let us fix such a face $\Delta$ of $K(\SS)$. Set $n^{\prime}:=\d \Delta$ and $f:=f_k \circ i_{\Delta} \colon \overline{T_{\Delta}} \longrightarrow \CC$. Note that we have $T_{\Delta}=\Spec (\CC [M(\SS\cap\Delta)])$ and $\overline{T_{\Delta}} =X(\SS \cap \Delta):=\Spec(\CC[\SS\cap\Delta])$. We shall calculate the function
\begin{equation}
\zeta_{g,\Delta}(t)=\zeta_{f,0}(\CC_{T_{\Delta} \cap W})(t) \in \CC(t)^*.
\end{equation}
For this purpose, we divide the polar cone
\begin{equation}
\Delta^{\vee}=\{ u \in \LL(\Delta)^* \ |\ \text{$\langle u,v \rangle \geq 0$ for any $v \in \Delta$}\}
\end{equation}
of $\Delta$ in $\LL(\Delta)^*$ by the equivalence relation
\begin{equation}
u \sim u^{\prime} \hspace{5mm}\overset{{\rm def}}{\Longleftrightarrow} \hspace{5mm}\Gamma(f_j|_{\Delta};u)=\Gamma(f_j|_{\Delta};u^{\prime}) \hspace{5mm}(\forall j \in I(\Delta) \sqcup \{k\})
\end{equation}
(see \eqref{eq:3-12}). Then we obtain a fan $\tl{\Sigma_{\Delta}}=\{ \sigma_i^{\prime}\}_i$ in $(\LL(\Delta)^*, M(\SS \cap \Delta)^*)$ such that $\bigsqcup_i \relint (\sigma_i^{\prime}) =\Delta^{\vee}$, where $\relint(\, \cdot\, )$ is the relative interior. Note that for the product $f_{\Delta}=\(\prod_{j \in I(\Delta)} f_j\) \cdot f_k \in \CC[\SS]$ and $u, u^{\prime} \in \Delta^{\vee}$ we have
\begin{equation}
u \sim u^{\prime} \hspace{5mm}\Longleftrightarrow \hspace{5mm} \Gamma(f_{\Delta}|_{\Delta};u)=\Gamma(f_{\Delta}|_{\Delta};u^{\prime}).
\end{equation}
By applying sufficiently many barycentric subdivisions to $\tl{\Sigma_{\Delta}}$, we obtain a fan $\Sigma:=\Sigma_{\Delta}=\{\sigma_i\}_{i \in I}$ in $(\LL(\Delta)^*, M(\SS \cap \Delta)^*)$ such that $\bigsqcup_i \relint (\sigma_i ) =\Delta^{\vee}$ and the (normal) toric variety $X_{\Sigma}$ associated with it is smooth. Then the open dense torus $T_{\Delta}^{\prime \prime}$ in $X_{\Sigma}$ is defined by
\begin{equation}
T_{\Delta}^{\prime \prime}=\Spec(\CC[M(\SS\cap\Delta)]).
\end{equation}
Since the subsemigroup $M(\SS \cap \Delta) \cap \Delta$ of $M(\SS \cap \Delta)$ is saturated, the affine toric variety $Z(\Delta):=\Spec(\CC[M(\SS\cap\Delta)\cap \Delta])$ is normal and there exists a natural $T_{\Delta}^{\prime \prime}$-equivariant morphism
\begin{equation}
\pi_1 \colon Z(\Delta) \longrightarrow X(\SS \cap \Delta)
\end{equation}
induced by $\SS \cap \Delta \subset M(\SS\cap\Delta)\cap\Delta$. There exists also a natural $T_{\Delta}^{\prime \prime}$-equivariant proper birational morphism
\begin{equation}
\pi_2 \colon X_{\Sigma} \longrightarrow Z(\Delta)
\end{equation}
induced by the morphism $\Sigma \longrightarrow \{\Delta^{\vee}\}$ of fans in $(\LL(\Delta)^*, M(\SS \cap \Delta)^*)$. Hence we obtain a $T_{\Delta}^{\prime \prime}$-equivariant morphism
\begin{equation}
\pi :=\pi_1 \circ \pi_2 \colon X_{\Sigma} \longrightarrow X(\SS \cap \Delta).
\end{equation}
We shall use this morphism $\pi$ for the calculation of $\zeta_{g,\Delta}(t)$. For a face $\tau \prec \Delta^{\vee}$ of $\Delta^{\vee}$, denote by $\Delta_{\tau}$ the polar face $\Delta \cap \tau^{\perp}$ of $\tau$ in $\LL(\Delta)$ and set
\begin{eqnarray}
T_{\tau}^{\prime}&:=&\Spec(\CC[M(\SS\cap\Delta)\cap\tau^{\perp}])=\Spec(\CC[M(\SS\cap\Delta)\cap\LL(\Delta_{\tau})]),\\
T_{\tau}&:=& \Spec(\CC[M(\SS \cap \Delta_{\tau})]).
\end{eqnarray}
Then $T_{\tau}^{\prime}$ and $T_{\tau}$ are $T_{\Delta}^{\prime \prime}$-orbits in $Z(\Delta)$ and $X(\SS \cap \Delta)$ respectively. Note that we have $T_{\{0\}}=T_{\Delta}$ and $T_{\{0\}}^{\prime}=T_{\Delta}^{\prime \prime}$ under this notation. Since the kernel of the canonical morphism
\begin{equation}
T_{\tau}^{\prime}=\Spec(\CC[M(\SS \cap \Delta) \cap \LL(\Delta_{\tau})]) \longrightarrow T_{\tau}=\Spec(\CC[M(\SS \cap \Delta_{\tau})])
\end{equation}
is isomorphic to the finite group $M(\SS \cap \Delta_{\tau})^*/( M(\SS \cap \Delta)\cap \LL(\Delta_{\tau}))^*$ (see \cite[p.22]{Oda}), the morphism $\pi_1|_{T_{\tau}^{\prime}} \colon T_{\tau}^{\prime} \longrightarrow T_{\tau}$ is a finite covering. 

For a cone $\sigma_i \in \Sigma$, denote by $T_{\sigma_i}^{\prime \prime}$ the $T_{\Delta}^{\prime \prime}$-orbit $\Spec(\CC[M(\SS\cap\Delta)\cap\sigma_i^{\perp}]) \simeq (\CC^*)^{\d \Delta-\d \sigma_i}$ in $X_{\Sigma}$ (in particular we have $T^{\prime \prime}_{\{0\}}=T_{\Delta}^{\prime \prime}$). Then we obtain a decomposition $X_{\Sigma}=\bigsqcup_{\sigma_i\in \Sigma} T_{\sigma_i}^{\prime \prime}$ of $X_{\Sigma}$. Note that the proper morphism $\pi_2\colon X_{\Sigma} \longrightarrow Z(\Delta)$ induces an isomorphism $\pi_2|_{T^{\prime \prime}_{\{0\}}} \colon T^{\prime \prime}_{\{0\}} \simto T^{\prime}_{\{0\}}$. 

Let us set $\tl{f}:=f \circ \pi \colon X_{\Sigma} \longrightarrow \CC$ and apply Proposition \ref{prp:2-9} to the constructible sheaf $\CC_{T_{\Delta}^{\prime \prime}\cap \pi^{-1}( T_{\Delta} \cap W)}\in \Dbc(X_{\Sigma})$. Since the morphism $\pi \colon X_{\Sigma} \longrightarrow X(\SS\cap\Delta)$ is proper and induces an isomorphism $\pi|_{T_{\Delta}^{\prime \prime}} \colon T_{\Delta}^{\prime \prime}\simto T_{\Delta}$, by the above descriptions of $\pi_1$ and $\pi_2$ we have
\begin{equation}
R\pi_*\(\CC_{T_{\Delta}^{\prime \prime}\cap \pi^{-1}(T_{\Delta} \cap W)}\)\simeq \CC_{T_{\Delta}\cap W}
\end{equation}
in $\Dbc(X(\SS\cap\Delta))$. Then by Proposition \ref{prp:2-9} for the calculation of
\begin{equation}
\zeta_{f,0}\(\CC_{T_{\Delta}\cap W}\)(t)=\zeta_{f,0}\(R\pi_*\(\CC_{T_{\Delta}^{\prime \prime}\cap \pi^{-1}(T_{\Delta} \cap W)}\)\)(t)\in \CC(t)^*
\end{equation}
it suffices to calculate the value of $\zeta_{\tl{f}}\(\CC_{T_{\Delta}^{\prime \prime}\cap \pi^{-1}(T_{\Delta} \cap W)}\)$ at each point of $\pi^{-1}(0)$.

Let $\sigma_0 \in \Sigma$ be a cone in $\Sigma$ such that $\relint(\sigma_0) \subset \Int(\Delta^{\vee})$ ($\Longleftrightarrow T_{\sigma_0}^{\prime \prime} \subset \pi^{-1}(0)$). In order to calculate the constructible function $\zeta_{\tl{f}}\(\CC_{T_{\Delta}^{\prime \prime} \cap \pi^{-1}(T_{\Delta} \cap W)}\)$ on $T_{\sigma_0}^{\prime \prime}$, take an $n^{\prime}$-dimensional cone $\sigma_1 \in \Sigma$ such that $\sigma_0 \prec \sigma_1$ and let $\{a_1,a_2,\ldots, a_{n^{\prime}}\}$ be the $1$-skelelton of $\sigma_1$. In other words, $a_i\neq 0\in M(\SS \cap \Delta)^*$ are the primitive vectors on the $1$-dimensional faces of $\sigma_1$. Set $p:=\d \sigma_0$. Without loss of generality, we may assume that $\{a_1,a_2,\ldots, a_p\}$ is the $1$-skeleton $\sigma_0$. For $j \in I(\Delta) \sqcup \{ k\}$, we set
\begin{equation}
m(j)_i:=\min_{w \in \Gamma_+(f_j)\cap \Delta} \langle a_i, w \rangle \geq 0 \hspace{5mm}(i=1,2,\ldots, n^{\prime}).
\end{equation}
For simplicity, we set also $m_i:=m(k)_i$ ($i=1,2,\ldots, n^{\prime}$).

Let $U_1 :=\CC^{n^{\prime}}(\sigma_1) \simeq\CC_y^{n^{\prime}}$ be the affine toric variety associated with the fan $\{\sigma^{\prime}\}_{\sigma^{\prime} \prec \sigma_1}$ in $(\LL(\Delta)^*, M(\SS \cap \Delta)^*)$. Then $U_1$ is an affine open subset of $X_{\Sigma}$ and in $U_1\simeq\CC_y^{n^{\prime}}$ the $T_{\Delta}^{\prime \prime}$-orbit $T_{\sigma_0}^{\prime \prime}$ is defined by
\begin{equation}
T_{\sigma_0}^{\prime \prime}=\{ y=(y_1,\ldots, y_{n^{\prime}})\ |\ y_1=\cdots =y_p=0, \ y_{p+1}, \ldots, y_{n^{\prime}}\neq 0\} \simeq (\CC^*)^{n^{\prime}-p}.
\end{equation}
Moreover for $j \in I(\Delta) \sqcup \{k\}$, on $U_1\simeq \CC_y^{n^{\prime}}$ the function $f_j \circ \pi \colon X_{\Sigma} \longrightarrow \CC$ has the form 
\begin{equation}
(f_j \circ \pi)(y)=y_1^{m(j)_1}y_2^{m(j)_2} \cdots y_{n^{\prime}}^{m(j)_{n^{\prime}}} \cdot f_j^{\sigma_1}(y),
\end{equation}
where $f_j^{\sigma_1}$ is a polynomial function on $U_1$. Note that by the assumptions $T_{\sigma_0}^{\prime \prime} \subset \pi^{-1}(0)$ and $f_k(0)=0$ we have $T_{\sigma_0}^{\prime \prime} \subset \tl{f}^{-1}(0)$. Since $f_k^{\sigma_1}|_{T_{\sigma_0}^{\prime \prime}} \not\equiv 0$ by the construction of the fan $\Sigma$, at least one of the integers $m_1,m_2,\ldots , m_p \geq 0$ must be positive. Note that by the definition of $I(\Delta)$ we have $f_j \circ i_{\Delta} \equiv 0$ on $X(\SS \cap \Delta)$ for any $j \notin I(\Delta) \sqcup \{k\}$. Hence we obtain our key formula
\begin{equation}
T_{\Delta}^{\prime \prime}\cap \pi^{-1}(T_{\Delta} \cap W)=T_{\Delta}^{\prime \prime} \cap \bigcap_{j \in I(\Delta)} \{ f_j^{\sigma_1} =0\}. 
\end{equation}
Moreover by the non-degeneracy of $(f_1,f_2,\ldots,f_k)$ (see Definition \ref{dfn:3-9}) we obtain the following lemma. 

\begin{lemma}\label{lem:4-1}
The gradient vectors $\{ \grad(f_j^{\sigma_1}) \}_{j \in I(\Delta)}$ (resp. $\{ \grad(f_j^{\sigma_1}) \}_{j \in I(\Delta) \sqcup \{k\}}$) are linearly independent on $\bigcap_{j \in I(\Delta)}\{f_j^{\sigma_1}=0\}$ (resp. $\bigcap_{j \in I(\Delta) \sqcup \{k\}} \{f_j^{\sigma_1}=0\}$) in a neighborhood of $T_{\sigma_0}^{\prime \prime}$.
\end{lemma}

\begin{proof}
First, by the non-degeneracy of $(f_1,f_2,\ldots, f_k)$, for any $u \in \Int(\Delta^{\vee}) \cap M(\SS \cap \Delta)^*$ the Laurent polynomials $L^{\prime}(f_j^u)$ on $\Spec(\CC[M(\SS)\cap\LL(\Delta)])\simeq (\CC^*)^{\d \Delta}$ defined by the $u$-parts $f_j^u \in \CC [\SS \cap \Delta]$ ($j \in I(\Delta) \sqcup \{k\}$) satisfy the conditions similar to the ones in Definition \ref{dfn:3-9}. Since the natural morphism
\begin{equation}
\Spec(\CC [M(\SS)\cap\LL(\Delta)])\simeq (\CC^*)^{\d \Delta}\longrightarrow T_{\Delta}=\Spec(\CC [M(\SS \cap \Delta)])
\end{equation}
is a finite covering, also the corresponding Laurent polynomials $L^{\prime \prime}(f_j^u)$ ($j \in I(\Delta) \sqcup \{k\}$) on $T_{\Delta}\simeq (\CC^*)^{\d \Delta}$ satisfy such conditions. Then the result follows from the classical arguments (see \cite{Oka-2} etc. for the detail). \qed
\end{proof}

For $j \in I(\Delta) \sqcup \{ k\}$, we set
\begin{equation}
h_j:=f_j^{\sigma_1}|_{T_{\sigma_0}^{\prime \prime}} \colon T_{\sigma_0}^{\prime \prime} \longrightarrow \CC. 
\end{equation}

\begin{proposition}\label{prp:4-2}
In the situation as above, we have
\begin{enumerate}
\item If $\d \sigma_0=1$, then for $y=(0,y_2,\ldots,y_{n^{\prime}})\in T_{\sigma_0}^{\prime \prime}\simeq (\CC^*)^{n^{\prime}-1}$ we have
\begin{equation}\label{eq:4-26}
\zeta_{\tl{f},y}\(\CC_{T_{\Delta}^{\prime \prime}\cap \pi^{-1}( T_{\Delta} \cap W)}\)(t)=\begin{cases}1-t^{m_1} & \text{if $y\in (\bigcap_{j\in I(\Delta)} \{h_j=0\}) \setminus \{h_k=0\} $},\\
1 &\text{otherwise}.
\end{cases}
\end{equation}
\item If $\d \sigma_0 \geq 2$, we have
\begin{equation}
\zeta_{\tl{f}}\(\CC_{T_{\Delta}^{\prime \prime}\cap \pi^{-1}(T_{\Delta} \cap W)}\) \Big|_{T_{\sigma_0}^{\prime \prime}}=\1_{T_{\sigma_0}^{\prime \prime}}.
\end{equation}
\end{enumerate}
\end{proposition}

\begin{proof}
Set $l:=\sharp \{ 1 \leq i \leq p \ |\ m_i>0\}$. If $l\geq 2$, then by Lemma \ref{lem:4-1} we obtain
\begin{equation}
\zeta_{\tl{f}}\(\CC_{T_{\Delta}^{\prime \prime} \cap \pi^{-1}(T_{\Delta} \cap W)}\)\Big|_{T_{\sigma_0}^{\prime \prime}} = \1_{T_{\sigma_0}^{\prime \prime}}
\end{equation}
(see for example \cite[p.48-49]{Oka-2} etc.). Let us consider the case where $l=1$. If $y \in T_{\sigma_0}^{\prime \prime} \setminus (\bigcap_{j\in I(\Delta)} \{h_j=0\})$ or $y \in (\bigcap_{j\in I(\Delta)} \{h_j=0\}) \cap \{h_k=0\}$, also by Lemma \ref{lem:4-1} we can show that $\zeta_{\tl{f},y}\(\CC_{T_{\Delta}^{\prime \prime}\cap \pi^{-1}( T_{\Delta} \cap W )}\)(t)=1$. By $\d T_{\Delta}^{\prime \prime}-\d T_{\sigma_0}^{\prime \prime}=p$, in a neighborhood of each point of $(\bigcap_{j\in I(\Delta)} \{h_j=0\}) \setminus \{h_k=0\}$, we have
\begin{equation}
\{ \tl{f} =\e \} \cap \( T_{\Delta}^{\prime \prime}\cap \pi^{-1}(T_{\Delta} \cap W)\) \simeq (\CC^*)^{p-1} \times A \qquad (0 <|\e | \ll 1), 
\end{equation}
where $A$ is a constructible set. If $p\geq 2$, the Euler characteristic of $(\CC^*)^{p-1}$ is zero and we can easily prove that the equality
\begin{equation}
\zeta_{\tl{f}}\(\CC_{T_{\Delta}^{\prime \prime} \cap \pi^{-1}(T_{\Delta} \cap W)}\)\Big|_{T_{\sigma_0}^{\prime \prime}} = \1_{T_{\sigma_0}^{\prime \prime}}
\end{equation}
holds (see \cite[p.48-49]{Oka-2} etc.). Finally, consider the case where $l=1$ and $p=1$. In this case, on $U_1 \simeq \CC_y^{n^{\prime}}$ the function $\tl{f}=f_k \circ \pi$ has the form
\begin{equation}
\tl{f}(y)=y_1^{m_1}\cdot (y_{2}^{m_{2}}y_{3}^{m_{3}}\cdots y_{n^{\prime}}^{m_{n^{\prime}}})\cdot f_k^{\sigma_1}(y).
\end{equation}
Then by Lemma \ref{lem:4-1}, for $y \in T_{\sigma_0}^{\prime \prime}$ we can easily prove \eqref{eq:4-26}.
\qed
\end{proof}

Now let us return to the proof of Theorem \ref{thm:3-12}. By the proposition above, in order to calculate $\zeta_{g,\Delta}(t)$, it suffices to consider the values of the $\CC(t)^*$-valued constructible function
\begin{equation}
\zeta_{\tl{f}}\(\CC_{T_{\Delta}^{\prime \prime}\cap \pi^{-1}(T_{\Delta} \cap W)}\)\Big|_{\pi^{-1}(0)} \colon \pi^{-1}(0) \longrightarrow \CC(t)^*
\end{equation}
only on $T_{\sigma_0}^{\prime \prime}$ for $\sigma_0 \in \Sigma$ such that $\relint(\sigma_0) \subset \Int(\Delta^{\vee})$ and $\d \sigma_0=1$. Let us take such a $1$-dimensional cone $\sigma_0 \in \Sigma$. Let $u\neq 0 \in M(\SS \cap \Delta)^*$ be the unique non-zero primitive vector on $\sigma_0$ and for $j \in I(\Delta) \sqcup \{k\}$ set $\gamma(f_j)_u^{\Delta}:=\Gamma(f_j|_{\Delta};u)$. Then $\gamma(f_j)_u^{\Delta}$ is naturally identified with the Newton polytope of the Laurent polynomial $h_j \colon T_{\sigma_0}^{\prime \prime} \longrightarrow \CC$ on $T_{\sigma_0}^{\prime \prime} \simeq (\CC^*)^{\d \Delta-1}$ and by Theorem \ref{thm:2-14} we have
\begin{eqnarray}
\lefteqn{(-1)^{\d \Delta-m(\Delta)}\chi\( \(\bigcap_{j \in I(\Delta)} \{h_j=0\}\) \setminus \{h_k=0\} \)} \nonumber\\
&=& (-1)^{\d \Delta-m(\Delta)}\left\{\chi\(\bigcap_{j \in I(\Delta)}\{h_j=0\}\)-\chi\(\bigcap_{j \in I(\Delta)\sqcup \{k\}}\{h_j=0\}\)\right\}\\
&=& \hspace*{-12mm}\sum_{\begin{subarray}{c}\alpha_1+\cdots +\alpha_{m(\Delta)}=\d \Delta-1 \\ \text{$\alpha_q \geq 1$ for $q \leq m(\Delta)-1$}, \ \alpha_{m(\Delta)} \geq 0 \end{subarray}} \hspace*{-12mm}\Vol_{\ZZ}(\underbrace{\gamma(f_{j_1})_u^{\Delta},\ldots, \gamma(f_{j_1})_u^{\Delta}}_{\text{$\alpha_1$-times}}, \ldots, \underbrace{\gamma(f_{j_{m(\Delta)}} )_u^{\Delta}, \ldots, \gamma(f_{j_{m(\Delta)}} )_u^{\Delta}}_{\text{$\alpha_{m(\Delta)}$-times}}). \hspace*{10mm}\label{eq:4-35}
\end{eqnarray}
Here we set $I(\Delta)\sqcup \{k\}=\{j_1,j_2,\ldots, k=j_{m(\Delta)}\}$. Now recall that we have
\begin{equation}
\Gamma(f_{\Delta}|_{\Delta};u)=\sum_{j \in I(\Delta) \sqcup\{k\}}\gamma(f_j)_u^{\Delta}.
\end{equation}
Hence if $\d \Gamma(f_{\Delta}|_{\Delta};u) <\d \Delta-1$, then all the mixed volumes in \eqref{eq:4-35} vanish and
\begin{equation}
\chi\(\(\bigcap_{j \in I(\Delta)} \{h_j=0\}\)\setminus \{h_k=0\} \)=0.
\end{equation}
This implies that for the calculation of $\zeta_{g,\Delta}(t)=\zeta_{f,0}\(\CC_{T_{\Delta} \cap W}\)(t)\in \CC(t)^*$, we have only to consider the compact faces $\gamma_1^{\Delta}, \gamma_2^{\Delta},\ldots , \gamma_{n(\Delta)}^{\Delta}$ of $\Gamma_+(f_{\Delta}) \cap \Delta$ such that $\d \gamma_i^{\Delta}=\d \Delta-1$ and their normal primitive vectors $u_1^{\Delta}, u_2^{\Delta}, \ldots, u_{n(\Delta)}^{\Delta} \in \Int(\Delta^{\vee}) \cap M(\SS \cap \Delta)^*$.

Summerizing these arguments, we finally obtain
\begin{eqnarray}
\zeta_{g,\Delta}(t)
&=& \zeta_{f,0}\(\CC_{T_{\Delta} \cap W}\)(t)\\
&=& \prod_{i=1}^{n(\Delta)} \(1-t^{d_i^{\Delta}}\)^{(-1)^{\d \Delta -m(\Delta)}K_i^{\Delta}}.
\end{eqnarray}
Since $K_i^{\Delta}=0$ for $m(\Delta)>\d \Delta$ by the definition of $K_i^{\Delta}$, we also have
\begin{equation}
\zeta_{g,0}(t)=\prod_{\begin{subarray}{c}\Gamma_+(f_k) \cap \Delta \neq \emptyset\\ \d \Delta \geq m(\Delta) \end{subarray}} \zeta_{g,\Delta}(t).
\end{equation}
This completes the proof of Theorem \ref{thm:3-12}. \qed

\section{Monodromy zeta functions of torus invariant sheaves}\label{sec:5}

In this section, we generalize our Theorem \ref{thm:3-4} to $T$-invariant constructible sheaves on general toric varieties.

First, let $X$ be a (not necessarily normal) toric variety over $\CC$ and $T\subset X$ the open dense torus which acts on $X$ itself. Let $X=\bigsqcup_{\alpha}X_{\alpha}$ be the decomposition of $X$ into $T$-orbits. 

\begin{definition}\label{dfn:5-1}
\begin{enumerate}
\item We say that a constructible sheaf $\F$ on $X$ is $T$-invariant if $\F|_{X_{\alpha}}$ is a locally constant sheaf of finite rank for any $\alpha$.
\item We say that a constructible object $\F \in \Dbc(X)$ is $T$-invariant if the cohomology sheaf $H^j(\F)$ of $\F$ is $T$-invariant for any $j \in \ZZ$. 
\end{enumerate}
\end{definition}

Note that the so-called $T$-equivariant constructible sheaves on $X$ are $T$-invariant in the above sense. 

From now on, we consider the (not necessarily normal) toric variety $X(\SS)$ and the regular function $f \colon X(\SS) \longrightarrow \CC$ on it considered in Section \ref{sec:3}. We shall freely use the notations in Section \ref{sec:3}. Let $\F\in \Dbc(X(\SS))$ be a $T$-invariant object. Our objective here is to calculate the monodromy zeta function
\begin{equation}
\zeta_f(\F)(t):=\zeta_{f,0}(\F)(t) \in \CC(t)^*
\end{equation}
of $\F \in \Dbc(X(\SS))$ at the $T$-fixed point $0 \in X(\SS)$. Since we have
\begin{equation}
\zeta_f(\F)(t)=\prod_{j \in \ZZ}\zeta_f(H^j(\F))^{(-1)^j},
\end{equation}
we may assume from the first that $\F$ is a $T$-invariant constructible sheaf on $X(\SS)$. For each face $\Delta \prec K(\SS)$ of $K(\SS)$, denote by $T_{\Delta} \subset X(\SS)$ the corresponding $T$-orbit in $X(\SS)$ and consider the decomposition
\begin{equation}
X(\SS)=\bigsqcup_{\Delta \prec K(\SS)} T_{\Delta}
\end{equation}
of $X(\SS)$ into $T$-orbits. For $\Delta \prec K(\SS)$, we denote the local system $\F|_{T_{\Delta}}$ on $T_{\Delta}$ by $\L_{\Delta}$. Let $j_{\Delta} \colon T_{\Delta} \longhookrightarrow X(\SS)$ be the inclusion. Then by Proposition \ref{prp:2-9} we have
\begin{equation}
\zeta_f(\F)(t)=\prod_{\Delta \prec K(\SS)} \zeta_f((j_{\Delta})_!\L_{\Delta})(t).
\end{equation}
In order to calculate the monodromy zeta functions $\zeta_f((j_{\Delta})_!\L_{\Delta})(t)\in \CC(t)^*$ as in the proof of Theorem \ref{thm:3-4}, we need the following elementary propositions.

\begin{proposition}\label{prp:5-2}
Let $\L$ be a local system of rank $r>0$ on $\CC^*=\CC\setminus \{0\}$. Denote by $A \in GL_r(\CC)$ the monodromy matrix of $\L$ along the loop $\{e^{\sqrt{-1}\theta}\ |\ 0 \leq \theta \leq 2\pi\}$ in $\CC^*$, which is defined up to conjugacy. Let $j \colon \CC^* \longhookrightarrow \CC$ be the inclusion.
\begin{enumerate}
\item Set $d:=\d {\rm Ker}(\id-A)$. Then we have
\begin{equation}
H^j(\CC^*;\L) \simeq \begin{cases} \CC^d & (j=0,1),\\ 0 & (\text{otherwise}). \end{cases}
\end{equation}
\item For any $j \in \ZZ$, we have
\begin{equation}
H^j(\CC;j_!\L)\simeq 0.
\end{equation}
\item Let $h$ be a function on $\CC$ defined by $h(z)=z^m$ ($m\in \ZZ_{>0}$) for $z \in \CC$. Then we have
\begin{equation}\label{eq:5-9}
\zeta_{h,0}(j_!\L)(t) =\det (\id -t^mA) \in \CC(t)^*.
\end{equation}
\end{enumerate}
\end{proposition}

\begin{proof}
For the proof of (i), see for example \cite[Lemma 3.3]{N-T}. The assertion (ii) is easily obtained from (i). Let us prove (iii). By taking small $\e >0$, for $k=0,1, \ldots , m-1$ set $p_k:=\e e^{\frac{2\pi \sqrt{-1} k}{m}} \in \CC^*$. Then we have an isomorphism 
\begin{equation}
\psi_h(j_! \L )_0 \simeq \bigoplus_{k=0}^{m-1}\L_{p_k}.
\end{equation}
We fix an isomorphism $\L_{p_0}\simeq \CC^r$ and for each $k=1,2, \ldots , m-1$ construct an isomorphism $\L_{p_k}\simeq \L_{p_0}=\CC^r$ by the translation of the sections of $\L$ along the path $\gamma_k\colon [0,1] \longhookrightarrow \CC^*$, $\gamma_k(s)=\e e^{\frac{2\pi \sqrt{-1} k}{m} s}$. Then we obtain an isomorphism 
\begin{equation}
\psi_h(j_! \L )_0 \simeq \bigoplus_{k=0}^{m-1}\L_{p_k} \simeq \CC^{mr}.
\end{equation}
Since the monodromy automorphism of $\psi_h(j_! \L )_0$ corresponds to the matrix 
\begin{equation}
\begin{pmatrix}
O & O & \ldots &O & A\\
\id & O & \ldots &O & O\\
O & \id & \ddots &\ddots & O\\
\vdots & \ddots  & \ddots & \ddots &\vdots \\
O & \ldots & \ldots & \id & O
\end{pmatrix}
\in GL_{mr}(\CC)
\end{equation}
via this isomorphism $\psi_h(j_! \L )_0 \simeq \CC^{mr}$, we obtain \eqref{eq:5-9}. \qed
\end{proof}

\begin{proposition}\label{prp:5-3}
Let $\L$ be a local system on $(\CC^*)^k$ and $j \colon (\CC^*)^k \longhookrightarrow \CC^k$ the inclusion. Let $h \colon \CC^k \longrightarrow \CC$ be a function on $\CC^k$ defined by $h(z)=z_1^{m_1}z_2^{m_2}\cdots z_k^{m_k} (\not\equiv 1)$ ($m_i \in \ZZ_{\geq 0}$) for $z\in \CC^k$. If $k \geq 2$, the monodromy zeta function $\zeta_{h,0}(j_!\L)(t)$ (resp. $\zeta_{h,0}(Rj_*\L)(t)$) of $j_!\L \in \Dbc(\CC^k)$ (resp. $Rj_*\L \in \Dbc(\CC^k)$) at $0 \in \CC^k$ is $1 \in \CC(t)^*$.
\end{proposition}

\begin{proof}
Since the proof of $\zeta_{h,0}(Rj_*\L)(t)\equiv 1$ is similar, we prove only $\zeta_{h,0}(Rj_!\L)(t)\equiv 1$. Let $F_0$ be the Milnor fiber of $h$ at $0 \in \CC^k$. Then there exist $\e_0, \eta_0 >0$ with $0< \eta_0 \ll \e_0$ such that the restriction
\begin{equation}
B(0; \e_0) \cap h^{-1}(D_{\eta_0}^*) \longrightarrow D_{\eta_0}^*
\end{equation}
of $h$ is a fiber bundle over the punctured disk $D_{\eta_0}^*=\{ x \in \CC \ | \ 0<|x|<\eta_0\}$ with fiber $F_0$. Furthermore, by using the special form of $h$, we may replace the above constant $\e_0, \eta_0 >0$ so that there exists also an isomorphism
\begin{equation}\label{isom}
\RG (h^{-1}(x) ; j_!\L ) \simto \RG (h^{-1}(x) \cap B(0; \e_0) ; j_!\L )
\end{equation}
for any $x \in D_{\eta_0}^*$. Indeed, this isomorphism can be obtained by applying Kashiwara's non-characteristic deformation lemma (\cite[Proposition 2.7.2]{K-S}) to the constructible sheaf $(j_!\L)|_{h^{-1}(x)}$ on the complex manifold $h^{-1}(x)$. Set $\F =(j_!\L)_{ \overline{B(0; \e_0)}} \in \Db(\CC^k)$. Then for any $j \in \ZZ$ the cohomology sheaf $H^j(Rh_*\F)$ is a local system on $D_{\eta_0}^*$ and via the isomorphism
\begin{equation}
H^j(\psi_h(j_! \L))_0 \simeq H^j(F_0 ;j_!\L) \simeq H^j(Rh_*\F)_x \qquad (x \in D_{\eta_0}^*) 
\end{equation}
(obtained by Proposition \ref{prp:2-7-2}) the monodromy automorphism of $H^j(\psi_h(j_!\L))_0$ corresponds to the one $Q_{j,x}\colon H^j(Rh_*\F)_x \simto H^j(Rh_*\F)_x$ obtained by the translation of the sections of the local system $H^j(Rh_*\F)_x$ along the path $\gamma_x \colon [0,1] \longrightarrow D_{\eta_0}^*$, $\gamma_x(s)=e^{2\pi \sqrt{-1} s}x$ (see the discussions just after \cite[Proposition 4.2.2]{Dimca}). This automorphism $Q_{j,x}$ can be functorially constructed as follows. First, define a morphism $\tl{\Psi}\colon [0,1] \times D_{\eta_0}^* \longrightarrow D_{\eta_0}^* $ by $\tl{\Psi}(s,x)=e^{2 \pi \sqrt{-1} s}x$ and let $\pi \colon [0,1] \times D_{\eta_0}^* \longrightarrow D_{\eta_0}^*$ be the projection. For $q=0,1$, let $i_q\colon D_{\eta_0}^* \simeq \{q\} \times D_{\eta_0}^* \longhookrightarrow [0,1] \times D_{\eta_0}^*$ be the inclusion and set $\Psi_q=\tl{\Psi}(q,*) =\tl{\Psi}\circ i_q\colon D_{\eta_0}^*\simto D_{\eta_0}^*$. Note that $\Psi_0=\Psi_1=\id_{D_{\eta_0}^*}$ in this case. Then for $q=0,1$ we obtain an isomorphism
\begin{equation}
R\pi_*\tl{\Psi}^{-1}Rh_*\F \simto R\pi_* (i_q)_*(i_q)^{-1} \tl{\Psi}^{-1}Rh_*\F \simeq \Psi_q^{-1}Rh_*\F
\end{equation}
in $\Dbc (D_{\eta_0}^*)$. Hence by setting $\Psi :=\Psi_1$ we obtain an automorphism of $Rh_*\F$
\begin{equation}
Rh_*\F \simto R\Psi_*\Psi^{-1}Rh_*\F \simto R\Psi_*\Psi_0^{-1}Rh_*\F =Rh_*\F
\end{equation}
which induces $Q_{j,x}$. Similarly, define a morphism $\tl{\Phi}\colon [0,1] \times \CC^k \longrightarrow \CC^k$ by
\begin{equation}
\tl{\Phi}(s,z)=\left(e^{\frac{2\pi \sqrt{-1} }{md}s}z_1, e^{\frac{2\pi \sqrt{-1}}{md}s}z_2, \ldots , e^{\frac{2\pi \sqrt{-1}}{md}s}z_l, z_{l+1}, \ldots , z_k\right)
\end{equation}
and let $\varpi \colon [0,1] \times \CC^k \longrightarrow \CC^k$ be the projection. For $q=0,1$, set $\Phi_q=\tl{\Phi}(q,*)\colon \CC^k \longrightarrow \CC^k$. In this case, we have $\Phi_0=\id_{\CC^k}$, and $\Phi_1$ induces the monodromy automorphisms of the global Milnor fiber $h^{-1}(x)$ and the local one $F_0 =h^{-1}(x)\cap B(0; \e_0)$ for any $x \in D_{\eta_0}^*$. Then by setting $\Phi :=\Phi_1\colon \CC^k \simto \CC^k$ we obtain also isomorphisms
\begin{equation}
\Phi^{-1}j_!\L \simot R\varpi_* \tl{\Phi}^{-1}j_!\L \simto \Phi_0^{-1}j_!\L =j_!\L
\end{equation}
and hence an automorphism $R_{j,x}$ of $H^j(F_0;j_!\L)$ defined by
\begin{equation}
H^j(F_0;j_!\L) \longrightarrow H^j(F_0; R\Phi_*\Phi^{-1}j_!\L)\simeq H^j(F_0; R\Phi_*\Phi_0^{-1}j_!\L)=H^j(F_0;j_!\L).
\end{equation}
By the functorial constructions of $Q_{j,x}$ and $R_{j,x}$, we can easily check that via the isomorphism $H^j(Rh_* \F)_x \simeq H^j(F_0;j_!\L)$ the automorphism $Q_{j,x}$ corresponds to $R_{j,x}$. Therefore, in order to calculate the monodromy zeta function $\zeta_{h,0}(j_!\L)(t)$ it suffices to calculate the one for the automorphisms $R_{j,x}\colon H^j(F_0;j_!\L) \simto H^j(F_0;j_!\L)$ induced by the isomorphism $\Phi^{-1}j_!\L \simto j_!\L $. Moreover, by the isomorphism \eqref{isom}, we have only to calculate the zeta function for the automorphisms of $H^j(h^{-1}(x) ;j_!\L)$ induced also by $\Phi^{-1}j_!\L \simto j_!\L$.

Now, without loss of generality, we may assume that there exists $1 \leq l \leq k$ such that $m_1, m_2, \ldots , m_l >0$ and $m_{l+1}= \cdots =m_k=0$. Moreover, by replacing exponents, we may assume also that $h(z)=(z_1^{m_1}\cdots z_l^{m_l})^m$ with $m\geq 1$ and ${\rm gcd}(m_1,\ldots,m_l)=1$. Set $d:=m_1+\cdots +m_l$. Let $M=(m_{i,j})\in {\rm SL}(l;\ZZ)$ be a unimodular matrix such that $m_{1,j}=m_j$ for $j=1,\ldots,l$. We define an isomorphism $\Lambda_M \colon (\CC^*)^l\times \CC^{k-l} \simto (\CC^*)^l\times \CC^{k-l}$ by
\begin{equation}
w=\Lambda_M(z)=(z_1^{m_{1,1}}\cdots z_l^{m_{1,l}}, \ldots, z_1^{m_{l,1}}\cdots z_l^{m_{l,l}}, z_{l+1}, \ldots, z_k)
\end{equation}
and an isomorphism $\Phi^{\prime} \colon \CC^k \simto \CC^k$ by
\begin{equation}
\Phi^{\prime}(w_1,\ldots,w_k)=\(e^{\frac{2\pi \sqrt{-1}}{m}}w_1,e^{\frac{2\pi \sqrt{-1} d_2}{md}}w_2,\ldots, e^{\frac{2\pi \sqrt{-1} d_l}{md}}w_l, w_{l+1},\ldots, w_k\),
\end{equation}
where we set $d_i:=m_{i,1}+\cdots+m_{i,l}$ for $i=2,\ldots ,l$. Moreover we define a function $h^{\prime} \colon (\CC^*)^l \times \CC^{k-l} \longrightarrow \CC$ by $h^{\prime}(w)=w_1^m$. Then we have
\begin{gather}
\Lambda_M(h^{-1}(x))=h^{\prime -1}(x)\qquad (x \in D_{\eta_0}^*), \\
\Lambda_M\circ (\Phi|_{(\CC^*)^l\times \CC^{k-l}})=(\Phi^{\prime}|_{(\CC^*)^l\times \CC^{k-l}}) \circ \Lambda_M.
\end{gather}

Let $j^{\prime} \colon (\CC^*)^k \longhookrightarrow (\CC^*)^l \times \CC^{k-l}$ be the inclusion and consider the local system $\L^{\prime}:=(\Lambda_{M}|_{(\CC^*)^k})_*\L$ on $(\CC^*)^k$. Similarly to $R_{j,x}$, we can construct an automorphism $R^{\prime}_{j,x}$ of $H^j(h^{\prime -1}(x);j^{\prime}_!\L^{\prime})$ by constructing an isomorphism $\Phi^{\prime -1}j^{\prime}_!\L^{\prime} \simto j^{\prime}_!\L^{\prime}$. Then via the natural isomorphism $H^j(F_0 ;j_!\L) \simeq H^j(h^{-1}(x);j_!\L) \simeq H^j(h^{\prime -1}(x);j^{\prime}_!\L^{\prime})$ induced by $\Lambda_M$ the automorphism $R_{j,x}$ corresponds to $R^{\prime}_{j,x}$. Define an automorphism $\Phi^{\prime \prime} \colon \CC^k \simto \CC^k$ by
\begin{equation}
\Phi^{\prime \prime}(w_1,\ldots,w_k)=\(e^{\frac{2\pi \sqrt{-1}}{m}}w_1,w_2,\ldots, w_k\). 
\end{equation}
Then we can construct also an automorphism $R^{\prime \prime}_{j,x}$ of $H^j(h^{\prime -1}(x);j^{\prime}_!\L^{\prime})$ by constructing an isomorphism $\Phi^{\prime \prime -1}j^{\prime}_!\L^{\prime} \simto j^{\prime}_!\L^{\prime}$. We define a homotopy morphism $\Theta \colon h^{\prime -1}(x) \times [0,1] \longrightarrow h^{\prime -1}(x)$ such that $\Theta(\cdot,0)= \Phi^{\prime \prime}|_{h^{\prime -1}(x)}$ and $\Theta(\cdot,1)=\Phi^{\prime}|_{h^{\prime -1}(x)}$ by 
\begin{equation}
\Theta(w,s)=\(e^{\frac{2\pi \sqrt{-1}}{m}}w_1,e^{\frac{2\pi \sqrt{-1}d_2}{md}s}w_2,\ldots, e^{\frac{2\pi \sqrt{-1} d_l}{md}s}w_l,w_{l+1}, \ldots, w_k\).
\end{equation}
Let $p\colon h^{\prime -1}(x) \times [0,1] \longrightarrow h^{\prime -1}(x)$ be the projection. Then similarly to $\Phi^{-1}j_!\L \simto j_!\L$, we can construct a natural isomorphism $\Theta^{-1}(j^{\prime}_!\L^{\prime}|_{h^{\prime -1}(x)}) \simto p^{-1}(j^{\prime}_!\L^{\prime}|_{h^{\prime -1}(x)})$. By applying Lemma \ref{lem:5-3-2} below to $Y=h^{\prime -1}(x)$ and $\Theta$, we get $R^{\prime}_{j,x}=R^{\prime \prime}_{j,x}$ and
\begin{equation}
\zeta_{h,0}(j_!\L)(t)=\prod_{j=0}^{\infty} \det(\id -tR^{\prime \prime}_{j,x})^{(-1)^j}.
\end{equation}
Note that each connected component $K$ of $h^{\prime -1}(x)$ is isomorphic to $(\CC^*)^{l-1} \times \CC^{k-l}$. Moreover by our assumption $k \geq 2$ and Proposition \ref{prp:5-2} the Euler-Poincar{\'e} index $\chi(\RG(K;j^{\prime}_!\L^{\prime}))$ of $\RG(K;j^{\prime}_!\L^{\prime})$ is zero. Then the result follows from the classical arguments as in \cite[Example (3.7)]{Oka-2} and the K{\" u}nneth formula. This completes the proof. \qed
\end{proof}

\begin{lemma}\label{lem:5-3-2}
Let $f_0,f_1 \colon Y \longrightarrow X$ be two morphisms of topological spaces. Set $I=[0,1]$ and let $p_Y \colon Y \times I \longrightarrow Y$ be the projection. Assume that there exists a homotopy morphism $\Theta \colon Y \times I \longrightarrow X$ between $f_0$ and $f_1$ such that $\Theta(\cdot ,q)=f_q$ for $q=0,1$. For $\F\in \Db(Y)$ and $\G \in \Db(X)$, assume that there exists an isomorphism $\Phi \colon \Theta^{-1}\G \simto p_Y^{-1}\F$. For $q=0,1$, let $f_q^{\sharp} \colon \RG(X;\G) \longrightarrow \RG(Y;\F)$ be the morphism obtained by
\begin{equation}
\xymatrix@C=12mm{
\RG(X;\G) \ar[r] &\RG(X; Rf_{q *} f_q^{-1}\G) \simeq\RG(Y;f_q^{-1}\G) \ar[r]^{\hspace*{20mm}\Phi |_{Y \times \{ q\} } } & \RG(Y;\F)}.
\end{equation}
Then we have $f_0^{\sharp}=f_1^{\sharp}$.
\end{lemma}

\begin{proof}
For $q=0,1$, let $i_q \colon Y \simeq Y \times \{ q\} \longhookrightarrow Y \times I$ be the embedding. Then we obtain the following commutative diagram:
\begin{equation}
\xymatrix@C=11mm@R=6mm{
\RG(X;\G) \ar[r] \ar[rd] & \RG(Y\times I; \Theta^{-1}\G) \ar[r]^{\Phi} \ar[d] & \RG(Y\times I; p_Y^{-1}\F) \ar[d] & \RG(Y;\F) \ar[l]_{\hspace{5mm}\sim}\\
& \RG(Y;i_q^{-1}\Theta^{-1}\G) \ar[r]^{\Phi |_{Y \times \{ q\} } } & \RG(Y;i_q^{-1}p_Y^{-1}\F). \ar@{=}[ru]}
\end{equation}
This proves the lemma.\qed
\end{proof}

With these propositions at hands, we can prove the following explicit formula for $\zeta_f(\F)(t)\in \CC(t)^*$ as in the same way as the proof of Theorem \ref{thm:3-4}. For each $\Delta \prec K(\SS)$ by fixing an isomorphism $M(\SS \cap \Delta) \simeq \ZZ^{\d \Delta}$, we obtain an isomorphism $T_{\Delta}=\Spec(\CC[M(\SS \cap \Delta)]) \simeq (\CC^*)^{\d \Delta}$. We regard $\L_{\Delta}$ as a local system of rank $r_{\Delta}$ on $(\CC^*)^{\d \Delta}$ via this isomorphism and denote by $A_j^{\Delta} \in GL_{r_{\Delta}}(\CC)$ ($j=1,2,\ldots, \d \Delta$) the monodromy matrices of $\L_{\Delta}$ along the loops
\begin{equation}
\{(1,1,\ldots, 1, \overset{j}{\check{e^{\sqrt{-1}\theta}}}, 1,\ldots, 1)\in (\CC^*)^{\d \Delta}\ |\ 0 \leq \theta \leq 2\pi\}
\end{equation}
in $(\CC^*)^{\d \Delta} \simeq T_{\Delta}$, which are determined up to conjugacy. Note that the matrices $A_1^{\Delta}, A_2^{\Delta}, \ldots, A_{\d \Delta}^{\Delta}$ mutually commute. Finally by using the inner conormal vectors $u_1^{\Delta}, u_2^{\Delta},\ldots, u_{n(\Delta)}^{\Delta} \in M(\SS \cap \Delta)^*$ of the compact faces $\gamma_1^{\Delta}, \gamma_2^{\Delta}, \ldots, \gamma_{n(\Delta)}^{\Delta}$ of $\Gamma_+(f) \cap \Delta$ introduced in Section \ref{sec:3} we set
\begin{equation}
B_i^{\Delta}:=\prod_{j=1}^{\d \Delta} (A_j^{\Delta})^{u_{i,j}^{\Delta}} \in GL_{r_{\Delta}}(\CC)
\end{equation}
for $1 \leq i \leq n(\Delta)$, where $(u_{i,1}^{\Delta}, u_{i,2}^{\Delta}, \ldots, u_{i,\d \Delta}^{\Delta})\in \ZZ^{\d \Delta}$ is the image of $u_i^{\Delta}$ by the isomorphism $M(\SS \cap \Delta)^*\simeq \ZZ^{\d \Delta}$.

\begin{theorem}\label{thm:5-4}
Assume that $f=\sum_{v \in \SS} a_v \cdot v \in \CC[\SS]$ is non-degenerate. Then the monodromy zeta function $\zeta_f(\F)(t)=\zeta_{f,0}(\F)(t)\in \CC(t)^*$ of the $T$-invariant constructible sheaf $\F$ at $0 \in X(\SS)$ is given by
\begin{equation}
\zeta_f(\F)(t)=\prod_{\Gamma_+(f) \cap \Delta \neq \emptyset} \zeta_{f,\Delta}(\F)(t),
\end{equation}
where for each face $\Delta \prec K(\SS)$ of $K(\SS)$ such that $\Gamma_+(f) \cap \Delta\neq \emptyset$ we set
\begin{equation}
\zeta_{f,\Delta}(\F)(t)=\prod_{i=1}^{n(\Delta)}\det (\id-t^{d_i^{\Delta}}B_i^{\Delta})^{(-1)^{\d \Delta-1}\Vol_{\ZZ}(\gamma_i^{\Delta})}.
\end{equation}
\end{theorem}

By the same methods, also for non-degenerate complete intersection subvarieties $W=\{f_1=\cdots =f_{k-1}=0\}\supset V=\{f_1=\cdots =f_{k-1}=f_k=0\}$ in $X(\SS)$ and $T$-invariant constructible sheaves $\F$ on $X(\SS)$ we can give a formula for the monodromy zeta function
\begin{equation}
\zeta_{f_k}(\F_W)(t):=\zeta_{f_k,0}(\F_W)(t) \in \CC(t)^*
\end{equation}
of $\F_W =\F \otimes_{\CC_{X(\SS)}} \CC_W \in \Dbc(X(\SS))$ at $0 \in X(\SS)$. The precise formulation is now easy and left to the reader.


\end{document}